\documentclass[12pt]{article}
\usepackage{fullpage,amsfonts,amsmath,amsthm,amssymb,mathrsfs,graphicx,upgreek}

\usepackage{tocloft}
\usepackage{enumerate}
\usepackage[colorlinks=true,linkcolor=blue]{hyperref} % added by Xavier

\usepackage[margin=2cm]{geometry}

  \usepackage[usenames,dvipsnames]{color}
 \newcommand{\be}{\begin{equation}}
\newcommand{\ee}{\end{equation}}
\newcommand{\bea}{\begin{eqnarray}}
\newcommand{\eea}{\end{eqnarray}}
\newcommand{\bean}{\begin{eqnarray*}}
\newcommand{\eean}{\end{eqnarray*}}

\newtheorem{theorem}{Theorem}[section]
\newtheorem{corollary}[theorem]{Corollary}

\newtheorem{lemma}[theorem]{Lemma}

\newtheorem{proposition}[theorem]{Proposition}

\newtheorem{claim}[theorem]{Claim}
\theoremstyle{definition}

\newtheorem{prop}[theorem]{Proposition}
\newcommand\remove[1]{{}}

\def\B{{\mathcal B}}

\def\G{{\mathcal G}}
\def\vG{{\vec{\mathcal G}}}
\def\O{{\mathcal O}}
\def\vO{{\vec{\mathcal O}}}
\def\vE{\vec E}
\def\po{\operatorname{\mathbf Po}}
\def\ex{\operatorname{\mathbb E}}
\def\Pr{\operatorname{\mathbb P}}

\def\Bin{\operatorname{\bf Bin}}   % BDM
\def\bfd{{\dvec}}
\def\bft{{\bf t}}
\def\bfs{{\bf s}}

\def\bfg{{\bf g}}
\def\bfx{{\bf x}}

\def\pr{{\mathbb P}}

\def\eps{\varepsilon}
\def\dfrac#1#2{\lower0.15ex\hbox{\large$\textstyle\frac{#1}{#2}$}}
\def\({\bigl(}
\def\){\bigr)}
\def\st{\mathrel{:}}

\def\P{\mathcal{P}}

\def\M{\mathcal{M}}
\def\calH{{\mathcal H}}

\def\Z{\mathcal{Z}}

\def\calH{\mathcal{H}}
\def\calB{\mathcal{B}}

\def\calB{\mathcal{B}}

\def\dvec{\boldsymbol{d}}

\def\Cov{\operatorname{Cov}}

\def\plus{+}

\def\var{{\bf Var}}

\newcommand{\ind}[1]{{\boldsymbol 1}_{\{{#1}\}}}
\newcommand{\eqn}[1]{\eqref{#1}}

\let\le=\leqslant
\let\leq=\leqslant
\let\ge=\geqslant

\numberwithin{equation}{section}

\title{Sandwiching between random regular graphs and Erd\H{o}s-R\'{e}nyi graphs: configuration model and unions of perfect matchings}
\author{Pu Gao \\ University of Waterloo \\ pu.gao@uwaterloo.ca \and Mikhail Isaev \\ University of New South Wales\\ m.isaev@unsw.edu.au \and Xavier Perez-Gimenez\thanks{Research supported in part by NSF grant DMS2201590.}  \\ University of Nebraska–Lincoln
\\ xperez@unl.edu}
\date{}

\begin{document}

\maketitle

\begin{abstract}
    We establish new couplings among several random graph and multigraph models related to the random regular graph $\G(n,d)$, including the configuration model and unions of random perfect matchings. As a main result, we verify the Kim–Vu sandwich conjecture for all large degrees $d=n-O(\log^4 n)$ and prove a weakened version for $d=O(\log^4 n)$, which are the only remaining open cases. Our approach introduces a  coupling framework that links $\G(n,d)$ and $\G(n,p)$ through a chain of intermediate models.
\end{abstract}

\section{Introduction}

The random regular graph model $\G(n,d)$ is obtained by sampling a $d$-regular graph on the vertex set $[n]$ uniformly at random. (We assume throughout the paper that $dn$ is always even.) In contrast to the Erd\H{os}-R\'{e}nyi graph $\G(n,p)$, the analysis of $\G(n,d)$ is considerably more challenging and relies on combinatorial enumeration techniques developed for graphs with a specified degree sequence. In the study of $\G(n,d)$, several random graph models have been proposed to approximate or simulate its properties. Among these, the most prominent is the configuration model introduced by Bollob\'{a}s~\cite{bollobas1980probabilistic}. In this model,
each vertex $i\in [n]$ is represented by a bin containing $d$ points. The configuration model $\P(n,d)$ is obtained by taking a uniform random perfect matching over the total $nd$ points, producing what is called a pairing, and two points that are matched to each other are called a pair in the pairing. Each pairing corresponds to a $d$-regular multigraph, obtained by contracting each bin into a single vertex: every matched pair of points $p,q$ gives rise to an edge $uv$ whenever $p$ belongs to vertex $u$ and $q$ belongs to vertex $v$. Let $G[P]$  denote the $d$-regular multigraph produced by pairing $P$. Note that $G[P]$ may contain both loops and multiple edges. In general, each $d$-regular multigraph may be obtained from a different number of pairings. However, an easy counting argument shows that every simple $d$-regular graph (i.e.~with no loops or multiple edges) corresponds to the same number of pairings. As a result, we obtain the following well-known result (see~e.g.~\cite{wormald1999models}) which we state below as Proposition~\ref{prop:1} after introducing some notation.

Let $\G_n$ and $\calH_n$ be sequences of probability measures on the same sequence of sample spaces $\Omega_n$. We say that $\G_n$ is contiguous with respect to $\calH_n$ (denoted by $\G_n \vartriangleleft \calH_n$) if for every sequence of events $A_n$, $\pr_{\calH_n}(A_n)\to 0$ implies that $\pr_{\G_n}(A_n)\to 0$. We say that $\G_n$ and $\calH_n$ are mutually contiguous if $\G_n \vartriangleleft \calH_n$ and $\calH_n \vartriangleleft \G_n$. We say $\G_n$ is equivalent to $\calH_n$ if $\pr_{\G_n}(A_n)=\pr_{\calH_n}(A_n)$ for all $A_n$, and we say $\G_n$ is asymptotically equivalent to $\calH_n$ if $\pr_{\G_n}(A_n)=\pr_{\calH_n}(A_n)+o(1)$ for all $A_n$. All asymptotics here and throughout the paper are taken with respect to $n$.

\begin{prop}[\cite{wormald1999models}]\label{prop:1}
$\G(n,d)$ is equivalent to $G[P]$ for $P\sim \P(n,d)$ conditional on the event that $G[P]$ is a simple graph.
\end{prop}
When $d=d(n)\to\infty$, the event that $G[P]$ is a simple graph has probability $o(1)$. However, for constant $d$, it follows from the analysis of the limiting distribution of the number of loops and double edges that $G[P]$ is simple with probability $\Omega(1)$. As a result, we obtain the following well-known fact as an immediate corollary of Proposition~\ref{prop:1}.
\begin{prop}[\cite{wormald1999models}] \label{p:configuration}
%\begin{enumerate}[(a)]
%\item $\G(n,d)$ is equivalent to $G[P]$ for $P\sim \P(n,d)$ conditional on the event that $G[P]$ is a simple graph.
%\item
For every fixed $d\ge 1$, $\G(n,d) $ is contiguous with respect to $G[P]$ for $P\sim \P(n,d)$. 
%\end{enumerate}
\end{prop}

We can analogously define the loopless configuration model $\P_*(n,d)$ as $\P(n,d)$ conditional on the event that no two points belonging to the same vertex are matched (i.e.~$G[P]$ has no loops for $P\sim \P(n,d)$). Obviously, Propositions~\ref{prop:1} and~\ref{p:configuration} hold as well if we replace $\P(n,d)$ by $\P_*(n,p)$.

We can yet define another random multigraph model from the superposition of perfect matchings. Given two (multi)graphs $G$ and $H$ defined on the same set of vertices, let $G+H$ denote the superposition of $G$ and $H$, i.e., for every pair of vertices $u,v$, the multiplicity of $uv$ in $G+H$ is equal to the sum of its multiplicities in $G$ and $H$ (the multiplicity of $uv$ is defined to be 0 if $u$ and $v$ are not adjacent). If $G$ and $H$ are simple graphs, let $G\cup H$ denote the union of $G$ and $H$ (that is, the union of their edge sets), and hence $G\cup H$ is again a simple graph.
Suppose that $n$ is even. Let $\M_d^\plus$ denote the random multigraph obtained by taking the superposition of $d$ independent uniformly random perfect matchings on $[n]$, and let $\M_d^\oplus$ denote $\M_d^\plus$ conditional on being simple. Let $\M_d^{\cup}$ denote the simple graph obtained by taking the union of $d$ independent uniformly random perfect matchings on $[n]$. In other words, $\M_d^{\cup}$ is obtained from $\M_d^{\plus}$ after replacing all the multiple edges by simple edges. The following proposition shows the relationship between $\G(n,d)$ and $\M_d^\oplus$ when $d$ is constant.
\begin{prop}[\cite{wormald1999models}]\label{p:matching}
 For every fixed $d\ge 3$, $\G(n,d) $ and $\M_d^\oplus$ are mutually contiguous. 
\end{prop}

Finally, Kim and Vu proposed to approximate $\G(n,d)$ by the Erd\H{o}s-R\'{e}nyi model $\G(n,p)$ for $p\approx d/n$. Of course $\G(n,p)$ and $\G(n,d)$ are supported on nearly disjoint set of graphs and thus it is not possible to expect any strong relation like asymptotic equivalence or contiguity. Kim and Vu conjectured that if $d\gg \log n$, then $\G(n,d)$ can be sandwiched between two copies of $\G(n,p)$ whose densities are both asymptotically $d/n$. This conjecture has recently been confirmed true for all $d$ such that $\min\{d,n-d\}\gg \log^4 n$.

\begin{prop}[\cite{gao2020kim}]\label{p:Kim-Vu}
    Suppose that $\min\{d,n-d\}\gg \log^4 n$. Then there exist $p_*, p^*=(1+o(1))d/n$, and a coupling $(G_*,G,G^*)$ such that
    \[
    G_*\sim \G(n,p_*),\quad G\sim \G(n,d),\quad G^*\sim \G(n,p^*),\quad \text{and}\ \pr(G_*\subseteq G\subseteq G^*)=1-o(1).
    \]
\end{prop}

The main objective of this paper is to establish connections between the aforementioned random graph and multigraph models. Our first result concerns the remaining open case of the Kim-Vu conjecture \footnote{During the preparation of this manusript, Behague, Il'kovi\v{c} and Montgomery announced that they have proved the conjecture in this open case.}. 

\begin{theorem}\label{thm:KM}
\begin{enumerate}[(a)]
    \item Proposition~\ref{p:Kim-Vu} holds for all $d=n-O(\log^4 n)$.
    \item Suppose that $d\gg \log n$ and $d=O(\log^4 n)$. There exist $p_*=(1-o(1))d/n$, $p^*=(8+o(1))d/n$ and a coupling $(G_*,G,G^*)$ such that
    \[
      G_*\sim \G(n,p_*),\quad G\sim \G(n,d),\quad G^*\sim \G(n,p^*),\quad \text{and}\ \pr(G_*\subseteq G\subseteq G^*)=1-o(1).
    \]
    \item If $d=o(\log n)$, then for every fixed $\eps>0$ there exists a coupling $(G,G^*)$ such that
    \[
      G\sim \G(n,d),\quad G^*\sim \G(n,(1+\eps)\log n/n),\quad \text{and}\ \pr (G\subseteq G^*)=1-o(1),\quad \text{if $n$ is even},
    \]
    and
    \[
      G\sim \G(n,d),\quad G^*\sim \G(n,(2+\eps)\log n/n),\quad \text{and}\ \pr (G\subseteq G^*)=1-o(1),\quad \text{if $n$ is odd}.
    \]
    \item If $d\le c\log n$ for some fixed $c>0$ then there exists $C(c)>0$ depending on $c$ only, such that the coupling in part (c) exists for $p^*=C(c)\log n/n$.
\end{enumerate}

\end{theorem}

We remark that part (a) of Theorem~\ref{thm:KM} settles the Kim–Vu conjecture in the regime of very large $d$, while part (b) treats the case $d=O(\log^4 n)$ and establishes a weaker form of the conjecture, where $p^*$
  differs from the conjectured value by only a constant factor. The coupling in part (c) is optimal when $n$ is even. For odd $n$, the additional $\log n / n$ term in the edge density of $G^*$ arises from a technical limitation in our proof, and we did not pursue an optimal coupling in this case.

Our proof takes a very different approach from all previous works~\cite{kim2004sandwiching,dudek2017embedding,gao2020sandwiching,gao2020kim}, relying on couplings between $\G(n,d)$ and other well-studied models, such as the configuration model and superpositions or unions of perfect matchings. We provide a brief overview of the proof of Theorem~\ref{thm:KM}, highlighting the intermediate coupling results obtained along the way, which are of independent interest and significance.

\subsection{Bridging by unions of perfect matchings}

The main goal is to embed $\G(n,d)$ into $\G(n,p)$ for some $p$ reasonably greater than $d/n$. Motivated by Proposition~\ref{p:matching}, we consider $\M_{d}^\cup$ as an intermediate model to bridge $\G(n,d)$ and $\G(n,p)$.

Frize and the third author~\cite{frieze2025threshold} showed that for every fixed $r\ge 3$, the union of random perfect $r$-matchings can be embedded into a random $r$-uniform hypergaph with independent hyperedges.  
Extending their proof, we show that such a coupling exists for graphs ($r=2$) as well.
\begin{theorem}\label{thm:Gnp-M}
There is a continuous and strictly increasing function $f:[1,\infty)\to[0,\infty)$ with $g(1)=0$ and $\lim_{x\to\infty}g(x)=1/4$ such that the following holds.
Let $n$ be even.
    Let $\eps>0$ be fixed, and $p=x\log n/n$ for $x=x(n)\ge 1+\eps$. Then, for any $d\le (g(x)+o(1))pn$ there exists a coupling $(H_*, G^*)$ such that
    \[
    H_*\sim \M_d^\cup, \quad G^*\sim \G(n,p),\quad \pr(H_*\subseteq G^*)=1-o(1).
    \]
\end{theorem}

The next step is to establish the other side of the coupling.

\begin{theorem}\label{thm:G-M} Let $n$ be even.
    Suppose that $d=\omega(1)$ and $d=o(n^{1/7}/\log n)$. Then, there exist $\tau(d)=(2+o(1))d$ and a coupling $(G_*,H^*)$ such that
    \[
    G_*\sim \G(n,d),\quad H^*\sim \M_{\tau(d)}^\cup, \quad \pr(G_* \subseteq H^*)=1-o(1).
    \]

\end{theorem}

The model $\M_d^\cup$ lends itself well to analysis due to the convenient arithmetic identity $\M_{a}^\cup \cup\M_b^\cup=\M_{a+b}^\cup$ for any integers $a,b\ge 0$, where the union on the left-hand side is taken over two independent copies, and equality is understood in the sense of distribution. A similar arithmetic property holds for $\G(n,p)$, commonly known as the sprinkling technique: to generate $\G(n,p)$, one may first expose $\G(n,p')$ for some $p' < p$, and then sprinkle an independent copy of $\G(n,q)$ on top, where the parameters satisfy $1-p=(1-p')(1-q)$. However, $\G(n,d)$ does not have such nice arithmetic identities. The closest known property of $\G(n,d)$ to the aforementioned arithmetic identities~\cite{hollom2025monotonicity}  is that when $d=\omega(1)$ and $d=o(n^{1/7})$, 
\begin{equation}
    d_{TV}(\G(n,d) , \G(n,t) \oplus \M_{d-t}^\oplus)=o(1), \quad \text{for any $\omega(1)=t\le d-1$}. \label{eq:decomposition}
\end{equation}
where $\G(n,t) \oplus \M_{d-t}^\oplus$ denotes the union of independent $\G(n,t)$ and $\M_{d-t}^\oplus$ conditioned on them being disjoint.

There are two major obstacles in proving Theorem~\ref{thm:G-M} by the decomposition~\eqref{eq:decomposition}. First, $\M_{d-t}^\oplus$ is not the union of independent matchings. It would take a large number, exponential in $d$, of independent matchings to contain a subset of $d$ pairwise disjoint matchings. Second, the decomposition~\eqref{eq:decomposition} requires $t=\omega(1)$, which prevents a complete decomposition of $\G(n,d)$ into $d$ perfect matchings, as required by Theorem~\ref{thm:G-M}, even setting aside the independence issue.

To overcome these two obstacles, we turn to the loopless configuration model, and use it to bridge $\G(n,d)$ and $\M_t^\cup$ for some $t\ge d$.

\subsection{Decomposition of the loopless configuration model}

Instead of embedding $\G(n,d)$ into $\M_{t}^\cup$ for some $t\ge d$, we aim instead to embed $\P_*(n,d)$ into $\M_t^\plus$ by repeately decomposing $\P_*(n,d)$ as in~\eqref{eq:decomposition}. We first deal with the ``decomposition'' for small $d$.
We defined the superposition of two (multi)graphs. The superposition of two pairings is defined in a natural way by taking the union of the points and the matchings of the two pairings. Note that a matching can be natually viewed as a pairing and hence $\M_d^{\plus}$ can be naturally viewed as the superposition of $d$ pairings. In what follows, when we write $\G_1+\G_2$ where $\G_1$ and $\G_2$ are two random (multi)graphs, it means the superposition of  independent $G_1\sim \G_1$ and $G_2\sim \G_2$.

\begin{theorem}\label{thm:smalld}Let $n$ be even.
Let  $\tau=\omega(1)$. Then, for every fixed  $d\ge 3$,   
\begin{enumerate}[(a)]
    \item 
there exists a coupling $(P,H)$ such that $P\sim \P_*(n,d)$, $H\sim \M^\plus_{\tau}$ and $\pr(P\subseteq H)=1-o(1)$;
\item there exists a coupling $(G,H)$ such that $G\sim \G(n,d)$, $H\sim \M^\cup_{\tau}$ and $\pr(G\subseteq H)=1-o(1)$. 
\end{enumerate}
\end{theorem}

Our next result shows that when $d=\omega(1)$, $\P_*(n,d)$ can be decomposed as $\P_*(n,d-1)+\M_1^+$, an analog of~\eqn{eq:decomposition}.

\begin{theorem} \label{thm:P-P}
    Let $d=o(n^{1/7}/\log n)$. Then, there exists a coupling $(P,P')$ such that
    $P\sim \P_*(n,d)+\M_1^\plus$, and $P'\sim \P_*(n,d+1)$ and
    \[
    \pr(P=P'')=1-O(\log d/d).
    \]
\end{theorem}

If the probability bound $O(\log d/d)$ were summable over $d$, then by repeatedly applying Theorem~\ref{thm:P-P} and then Theorem~\ref{thm:smalld}, we could have obtained a coupling which embeds $\P_*(n,d)$ into $\M_{(1+o(1))d}^+$.  
However, the probability bound $O(\log d/d)$ is too large for a union bound.
Moreover, the probability bound $O(\log d/d)$ is optimal from our proof method, and we think that it might be the best possible among all couplings. 
To overcome this, we use an extra matching, and show that $\P_*(n,d+1)$ can be embedded into $\P_*(n,d)+\M_2^\plus$.

\begin{theorem}\label{thm:P-P2} Let $n$ be even. Suppose that $d=o(n^{1/7}/\log n)$.
    There is a coupling $(P,P')$ such that $P\sim \P_*(n,d)+\M_2^\plus$,
    $P'\sim \P_*(n,d+1)$, and $\pr(P'\subseteq P)= 1-O(\sqrt{\log d}/d^{1.5})$.
\end{theorem}

As a corollary of Theorems~\ref{thm:smalld} and~\ref{thm:P-P2}, we obtain the following result about the relation between the two models $\P_*(n,d)$ and $\M_t^\plus$, as desired.

\begin{theorem}\label{thm:P-M} Let $n$ be even.
    Suppose that $d=\omega(1)$ and $d=o(n^{1/7}/\log n)$. There exist $\tau(d)=(2+o(1))d$, and a coupling $(P,H)$ such that
    \[
    P\sim \P_{*}(n,d),\quad H\sim \M_{\tau(d)}^\plus,\quad \pr(P\subseteq H)=1-o(1).
    \]
\end{theorem}

\subsection{Embedding $\G(n,d)$ into $\P_*(n,d+1)$ } 

Given Theorem~\ref{thm:P-M}, it only requires a coupling between $\G(n,d)$ and $\P_*(n,t)$ to complete the final step in the proof of Theorem~\ref{thm:G-M}. This coupling is given below.
\begin{theorem}\label{thm:P-G} Let $n$ be even.
    Suppose that $d=o(n^{1/7}/\log n)$. Then, there is a coupling $(G_*,P^*)$ such that
    \[
    G_*\sim \G(n,d),\quad P^*\sim \P_*(n,d+1), \quad \pr(G_* \subseteq G[P^*])=1-O(\sqrt{\log d}/d^{1.5}).
    \]
    
\end{theorem}

\section{Tools}

Given two degree sequences $\bfd,\bfs$ defined on the same set of vertices, we say $\bfd\le \bfs$ if $\bfd_i\le \bfs_i$ for all $i$.  Given a graph $H$, let $\bfd^H$ denote the degree sequence of $H$. 

Let ${\mathbb N}$ denote the set of positive integers.
Let $U$ and $V$ be two disjoint sets. Let $\bfs\in {\mathbb N}^U$ and $\bft\in {\mathbb N}^V$, and let $\B(\bfs,\bft)$ be a uniform random bipartite graph on $U\cup V$ with degree sequence $(\bfs,\bft)$. Let $\Delta(\bfs)$ denote the maximum component in $\bfs$. Without loss of generality, we may assume that vertices in $U$ and $V$ are ordered so that $s_1\ge s_2\ge \cdots\ge s_{|U|}$ and $t_1\ge t_2\ge \cdots\ge t_{|V|}$. With $\bfd=(\bfs,\bft)$, let $ M(\bfd) =\sum_{u\in U}s_u=\sum_{v\in V}t_v$ and $J(\bfd) = \sum_{i=1}^{s_1} t_i + \sum_{j=1}^{t_1} s_j$. It follows that every bipartite graph $G$ on $U\cup V$ with degree sequence $(\bfs,\bft)$ has exactly $M$ edges. Moreover, for every $v$ in $G$, the number of 2-path starting at $v$ is at most $J(\bfd)$. Let $N(\bfs,\bft)$ denote the number of bipartite graphs with degree sequence $(\bfs,\bft)$.
The following estimation of $N(\bfs,\bft)$ was given by McKay~\cite{mckay1984asymptotics}.

\begin{theorem}\label{thm:bipartiteenum}
    Suppose that $M(\bfd)\to\infty$ where $\bfd=(\bfs,\bft)$ is a bipartite degree sequence, and $1\le \Delta(\bfd)^2<cM(\bfd)$ for some constant $c<1/6$. Then,
    \[
    N(\bfs,\bft)=\frac{M(\bfd)!}{\prod_{i}^{|U|}s_i! \prod_{j=1}^{|V|} t_j!} \exp\left(-\frac{\sum_{i=1}^{|U|}s_i(s_i-1) \sum_{j=1}^{|V|}t_j(t_j-1)}{2 M(\bfd)^2}+O(\Delta(\bfd)^4/M(\bfd))\right).
    \]
\end{theorem}

Given a bipartite graph $H$ on $U\cup V$, let $H^+$ and $H^-$ denote the events that $H\subseteq \B(\bfs,\bft)$ and $H\cap \B(\bfs,\bft)=\emptyset$ respectively. Let $e(H)$ and $\Delta(H)$ denote the number of edges, and the maximum degree of $H$. The conditional edge probability $\pr\big(uv\in \B(\bfs,\bft)\mid H_1^+,H_2^-\big)$ was estimated in~\cite[Theorem 1]{gao2023subgraph}. In this paper we need a bipartite version of it, given below. 
\begin{theorem}\label{thm:bipartiteEdgeProb}
    Let $H_1$ and $H_2$ be two disjoint bipartite graphs on $U\cup V$. Suppose that $\bfd^{H_1}\le \bfd$ where $\bfd:= (\bfs,\bft)$, and $uv\in (U\times V)\setminus (H_1\cup H_2)$. Then,
    \begin{align*}
   \pr\big(uv\in \B(\bfs,\bft)\mid H_1^+,H_2^-\big)= \frac{(d_u-d^{H_1}_u)(d_v-d^{H_1}_v)}{M-e(H_1)}\left(1+O\left(\xi\right)\right),
    \end{align*}
    where 
    \begin{align*}
        \xi &= \frac{J(\bfd)+\Delta(\bfd)(1+\Delta(H_2))+(d_u-d^{H_1}_u)(d_v-d^{H_1}_v)}{M-e(H_1)}+\frac{e(H_2)\Delta(\bfs)\Delta(\bft)}{(M-e(H_1))^2}.
    \end{align*}
    provided that $\xi=o(1)$.
\end{theorem}

{\bf Remark}. It is possible and fairly straightforward to specify the coefficients in $O(\xi)$ (different coefficients for the upper and the lower bounds) as in~\cite[Theorem 1]{gao2023subgraph}. But for simplicity we state the slightly weaker probability bound for $\pr\big(uv\in \B(\bfs,\bft)\mid H_1^+,H_2^-\big)$ which is sufficient for the purpose of this paper.

\proof The proof follows exactly the proof of~\cite[Theorem 1]{gao2023subgraph} by defining, and estimating the number of, switchings between the set of bipartite graphs in the probability space $\G(\bfd)$, conditioned to $H_1^+,H_2^-$ and containing the edge $uv$, and the set of bipartite graphs in the same probability space, conditioned to $H_1^+,H_2^-$ but to the graphs that do not contain the edge $uv$. The switchings are defined exactly the same as in~\cite{gao2023subgraph}, except that the vertices are chosen with respect to the vertex bipartition. More precisely, the 4-tuple $(x,a,y,b)$ specified by a switching must satisfy $x,b\in V$ and $a,y\in U$. Then, the rest of the proof of~\cite[Theorem 1]{gao2023subgraph} follows exactly the same by estimating $f$ and $b$, the number of forward and backward switchings, with~\cite[Claim 14]{gao2023subgraph} replaced by the following new (and less accurate) bound
\[
(M-e(H_1))^2\left(1-C\xi\right) \le f(G)\le (M-e(H_1))^2,
\]
 and with~\cite[Claim 15]{gao2023subgraph} replaced by
\[
(d_u-d^{H_1}_u)(d_v-d^{H_1}_v)(M-e(H_1))(1-C\xi)\le b(G')\le (d_u-d^{H_1}_u)(d_v-d^{H_1}_v)(M-e(H_1)),
\]
where $C>0$ is some sufficiently large constant.  
\qed 

We also need the following estimate by McKay~\cite{mckay1985asymptotics} on the number of graphs with a given degree sequence avoiding a specified set of edges. Given a nonbipartite degree sequence $\bfd=(d_1,\ldots,d_n)$, let $M(\bfd)=\sum_{i=1}^n d_i$.
\begin{theorem}\label{thm:subgraphenum}
    Let $\bfg=(g_1,\ldots,g_n)$ be a degree sequence on $[n]$ and $X$ be a simple graph on $[n]$ with degree sequence $\bfx$. Suppose $\Delta(\bfg)\ge 1$, $\hat\Delta(\bfg)=o(M(\bfg))$ where $\hat\Delta(\bfg)=\Delta(\bfg)^2+\Delta(\bfg)\Delta(\bfx)$. Define
    \[
    \lambda(\bfg)=\frac{1}{2M(\bfg)}\sum_{j=1}^n g_i(g_i-1),\quad \mu(\bfg,X)=\frac{1}{M(\bfg)}\sum_{ij\in X} g_ig_j.
    \]
   Let $N(\bfg,X)$ denote the number of simple graphs with degree sequence $\bfg$ and which no edge in common with $X$. Then,
    \[
    N(\bfg,X)=\frac{M(\bfd)!}{(M(\bfd)/2)!2^{M(\bfd)/2}\prod_{i=1}^n g_i!} \exp(-\lambda(\bfg)-\lambda(\bfg)^2-\mu(\bfg,X)+O(\hat\Delta(\bfg)^2/M(\bfg))).
    \]
\end{theorem}

The next lemma will be helpful to construct couplings of variables.

Suppose $X\sim (\Omega_X,\mu_X)$ and $Y\sim (\Omega_Y,\mu_Y)$ are two random variables from with finite sample spaces $\Omega_X$ and $\Omega_Y$. Let $\B\subseteq \Omega_X\times \Omega_Y\in $.
We refer to Koperberg \cite[Proposition 6]{Koperberg2024} for the proof of the following result.

\begin{theorem}[Strassen’s Theorem with deficiency \cite{Koperberg2024}]\label{T:couplingNew}
Let $\epsilon \in [0,1]$.
There exists a coupling $(X,Y)\in \Omega_X\times \Omega_Y$ with marginal distributions $\mu_X$ and $\mu_Y$ such  
that
$\Pr(XY \in \calB)\leq \epsilon$   if and only if 
\begin{equation}\label{eq:Hall}
    \mu_X(A) \leq \mu_Y(N(A)) + \epsilon \quad \text{for any $A \subseteq \Omega_X$,}
\end{equation}
where $N(A) = \{y\in \Omega_Y \st \text{there is $x\in \Omega_x$ that } (x,y) \notin \calB\}$. 
\end{theorem}

The following corollary will be convenient to establish the existence of a required coupling once we find an almost suitable coupling.    
For our purposes, it is sufficient to consider discrete random variables, but it can be generalised to separable metric spaces since Theorem \ref{T:couplingNew} can be derived from \cite[Theorem 11.6.3]{Dudley2002}.

\begin{corollary}\label{lem:coupling}
    Suppose that $P$ is a distribution on $\Omega_X\times \Omega_Y$ such that $P(\B)\le \alpha$, $d_{TV}(\mu_X,P_X)\le \beta$, and $d_{TV}(\mu_Y,P_Y)\le \gamma$, where $P_X$ and $P_Y$ are the marginal distributions of $P$. Then, there exists a distribution $\pi$ on $\Omega_X\times \Omega_Y$ such that $\pi_X=\mu_X$, $\pi_Y=\mu_Y$ and $\pi(\B)\le \alpha+\beta+\gamma$.
\end{corollary}
\begin{proof}
 Consider any $A\subseteq \Omega_X$. Using Theorem \ref{T:couplingNew} and assumptions, we can bound
 \[
    \mu_X (A) \leq P_{X}(A)+ \beta \leq P_{Y}(N(A)) + \beta + \alpha \leq \mu_{Y}(N(A)) + \alpha +\beta + \gamma.
 \]
 Applying now Theorem \ref{T:couplingNew} in the opposite direction, we complete the proof.
\end{proof}

The following is another corollary of Theorem~\ref{T:couplingNew} whose proof can be found in~\cite{isaev2025sprinkling}.

\begin{corollary}[Corollary 2.2 of~\cite{isaev2025sprinkling}]
    \label{cor:coupling2} Let $D$ be a bipartite graph on $(S,T)$ and $|D|$ be the number of edges in $D$.
    Let $\delta,\eps\in [0,1]$ and 
    \begin{align*}
        S_{\text{good}} &= \left\{x\in S: \text{deg}_D(x)\ge (1-\eps)\frac{|D|}{|S|}\right\}\\
        T_{\text{good}} &= \left\{y\in T: \text{deg}_D(y)\ge (1-\eps)\frac{|D|}{|T|}\right\}.
    \end{align*}
    Assume that $|S_{\text{good}}|\ge (1-\delta)|S|$ and $|T_{\text{good}}|\ge (1-\delta)|T|$. Then, there is a coupling $(X,Y)$ such that $X$ and $Y$ are uniformly distributed on $S$ and $T$, respectively, and
    \[
    \pr(XY\notin D)\le 2\delta+\frac{\eps}{1-\eps}.
    \]
\end{corollary}

\section{Small $d$: proof of Theorem~\ref{thm:smalld}}

Recall that two sequences of probability measures $\mu_n$ and $\pi_n$ defined on the same sample space are called mutually contiguous if for any sequence of events $A_n$, $\lim_{n\to\infty}\mu_n(A_n)=0$ if and only if 
$\lim_{n\to\infty}\pi_n(A_n)=0$. We use $\G_n \vartriangleleft \vartriangleright  \calH_n$ to denote that $\G_n$ and $\calH_n$ are mutually contiguous.

\begin{lemma} Let $n$ be even.
    For every fixed $d\ge 3$, $\P_*(n,d) \vartriangleleft \vartriangleright  \M_{d}^\plus$.  
\end{lemma}

\proof Theorems 4.7 and 4.8 of~\cite{wormald1999models} hold for $\P_*(n,d)$ instead of $\G_{n,d}$ following exactly the same proof via small subgraph conditioning, by noting that $\G_{n,d}$ is the pairing model conditioned to no loops or double edges, whereas $\P_*(n,d)$ is the pairing model conditioned to no loops. Then,~\cite[Theorem 4.15]{wormald1999models} after $\G(n,d)$ is replaced replaced by $\P_*(n,d)$ and $\oplus$ (superpositions restricted to simple graphs) is replaced by $+$ (superpositions). In particular, $\P_*(n,d) \vartriangleleft \vartriangleright  \P_*(n,d-1)+\M_1^\plus$. This immediately yield our lemma by induction on $d$. \qed 

Given a probability measure $\G_n$, let $\Omega(\G_n)$ denote the sample space of $\G_n$. Thus, $\Omega(\P_*(n,d))$ denotes the set of pairings that can be produced by the pairing model $\P_*(n,d)$.

\begin{lemma}\label{lem:typical} Let $n$ be even.
    Let $d\ge 3$ be fixed, and let $f_n$ be an arbitrary functions of $n$ such that $f_n=\omega(1)$. Then, with $\Omega=\Omega(\P_*(n,d))$, for every $\eps>0$,
        $|\Omega_{f}|\ge (1-\eps) |\Omega|$,
        where
        \[
\Omega_{f}=\left\{P\in \Omega:        \frac{f_n^{-1}}{|\Omega|}\le \pr_{\M_{d}^\plus}(P) \le \frac{f_n}{|\Omega|}\right\}.
        \]
\end{lemma}

\proof Let 
\begin{align*}
    \Omega^-_{f} &=\left\{P\in \Omega:         \pr_{\M_{d}^\plus}(P) <\frac{f_n^{-1}}{|\Omega|}\right\}\\
    \Omega^+_{f} &=\left\{P\in \Omega:         \pr_{\M_{d}^\plus}(P) >\frac{f_n}{|\Omega|}\right\}
\end{align*}
It suffices to show that $|\Omega^-_f|, |\Omega^+_f|<(\eps/2)|\Omega|$.
Suppose on the contrary that $|\Omega^-_f|\ge (\eps/2)|\Omega|$. 
Then, $\pr_{\P_*(n,d)}(\Omega^-_f)\ge \eps/2$. However, 
\[
\pr_{\M_{d}^\plus}(\Omega^-_f) =\sum_{P\in \Omega^-_f} \pr_{\M_{d}^\plus}(P)< \sum_{P\in \Omega^-_f} \frac{f_n^{-1}}{|\Omega|} =f_n^{-1} \cdot \pr_{\P_*(n,d)}(\Omega_f) \le f_n^{-1}=o(1),
\]
contradicting with $\P_*(n,d)\vartriangleleft \vartriangleright  \M_{d}^\plus$. Hence, $||\Omega^-_f|<(\eps/2)|\Omega|$. On the other hand, if $|\Omega^+_f|\ge (\eps/2)|\Omega|$ then $\pr_{\P_*(n,d)}(\Omega^+_f)\ge \eps/2$ and thus
\[
\pr_{\M_{d}^\plus}(\Omega^+_f) =\sum_{P\in \Omega^+_f} \pr_{\M_{d}^\plus}(P)> \sum_{P\in \Omega^+_f} \frac{f_n}{|\Omega|} =f_n \cdot \pr_{\P_*(n,d)}(\Omega^+_f) \ge f_n\eps/2>1,
\]
contradicting with $\pr_{\M_{d}^\plus}(\Omega^+_f)\le 1$. Thus, $|\Omega^+_f|<(\eps/2)|\Omega|$ as desired. \qed

\remove{
%%%%%%%%%%
%%%%%%%%%%
\begin{theorem}\label{thm:Pairing-Matching}
   Let  $\tau=\omega(1)$. Then, for every fixed  $d\ge 1$,   there exists a coupling $(P,H)$ such that $P\sim \P_*(n,d)$, $M\sim \M^\plus_{\tau}$ and $\pr(P\subseteq M)=1-o(1)$. 
\end{theorem}
}

\medskip

{\em Proof of Theorem~\ref{thm:smalld}} Parts (a,b) follow by an almost identical proof. Thus, we first present the proof for part (a).
Set 
\[
i^* = \lfloor \tau/d \rfloor.
\]
Let $f_n$ be any function of $n$ satisfying
\[
f_n=\omega(1),\quad \left(1-\frac{1}{2}f_n^{-2}\right)^{i^*}<f_n^{-1}.
\]
Since $\tau=\omega(1)$ such a function exists.
Let $\Omega_{f}$ be defined as in Lemma~\ref{lem:typical}  with $f_n$ chosen above.

Let $H_1,\ldots, H_{i^*}$ be independent copies of $\M^{\plus}_{d}$ and let $H=\cup_{i=1}^{i^*} H_i$. Then, $\hat H\sim \M_{di^*}^{\plus} \subseteq M^+_{\tau}$.

For every $P\in \Omega_f$, let $w(P)=|\Omega|\cdot\pr_{\M_d^\plus}(P)$. By definition of $\Omega_f$,
\begin{equation}
f_n^{-1}\le w(P) \le f_n, \quad \text{for all $P\in \Omega_f$}.\label{eq:weight}
\end{equation}

Construct $\hat P\in \Omega$ by the following procedure, which starts from step $i=1$, and either sets $\hat P$ to be $H_i$, ore rejects $H_i$ and proceed to the next step $i+1$, until reaching the final step $i^*$.

\begin{enumerate}[(i)]
    \item If $H_i\notin \Omega_{f}$, reject $H_i$. Otherwise, accept $H_i$ with probability $f_n/w(H_i)$, and reject it with the remaining probability.
    \item If $H_i$ is accepted, let $\hat P=H_i$ and terminate the procedure. 
    If $H_i$ is rejected, and $i<i^*$, set $i$ to be $i+1$ and go back to step (i). If $H_i$ is rejected, and $i=i^*$, set $\hat P=\emptyset$ and terminate.
\end{enumerate}

If $\hat P\neq \emptyset$ then obviously $\hat P\subseteq \hat H$. Next, we claim that with high probability, $\hat P\neq \emptyset$ and moreover, if $\hat P\neq \emptyset$ and $\hat P$ is uniform in $\Omega_f$.
\begin{claim}\label{claim:marginal}
    \begin{enumerate}[(a)]
        \item $\pr(\hat P=\emptyset)=o(1)$, and $\pr(\hat P=P \mid \hat P\neq \emptyset)=|\Omega_f|^{-1}$ for all $P\in \Omega_f$;
        \item $d_{TV}(\hat P, \P_*(n,d))=o(1)$.
    \end{enumerate}
\end{claim}
Then the assertion of the theorem follows by Lemma~\ref{lem:coupling} and Claim~\ref{claim:marginal}(b). It only remains to prove Claim~\ref{claim:marginal}.

{\em Proof of Claim~\ref{claim:marginal}.\ } Note that $\Omega(\hat P)=\Omega(\P_*(n,d))\cup \{\emptyset\}$. We may equivalently extend $\Omega(\P_*(n,d))$ to include the element $\emptyset$, so that $\hat P$ and $\P_*(n,d)$ are defined on the same sample space, where $\P_*(n,d)$ has probability zero on $\emptyset$, and is uniform on all the other elements. Then, part (b) follows by part (a) and Lemma~\ref{lem:typical} by sending $\eps$ to zero.

For part (a), $\hat P=\emptyset$ occurs when rejections occur for all the $i^*$ steps. In each step $i$, $H_i$ is accepted with probability at least $(1-\eps)f_n^{-2}$, where $1-\eps$ accounts for a lower bound for the probability that $H_i\in\Omega_f$, and by~\eqref{eq:weight} ,$f_n^{-2}$ is a lower bound for the probability that $H_i$ is accepted given that $H_i\in\Omega_f$. It follows then that
\[
\pr(\hat P=\emptyset) \le \left(1-(1-\eps)f_n^{-2}\right)^{i^*} \le \left(1-\frac{1}{2}f_n^{-2}\right)^{i^*}\le f_n^{-1}=o(1),
\]
by the choice of $i^*$.

For the second assertion of part (a), let $P\in \Omega_{f}$. Then,

\begin{align*}
    \pr(\hat P=P) =& \sum_{i=1}^{i^*} \left(\sum_{h_1,\ldots,h_{i-1}\in \Omega} \prod_{j=1}^{i-1} \pr(H_i=h_i)\left(\ind{h_i\notin\Omega_f}+\left(1-\frac{f_n}{w(h_i)}\right)\ind{h_i\in\Omega_f}\right) \right)\pr(H_i=P) \frac{f_n}{w(P)}\\
    =& \sum_{i=1}^{i^*} \eta^{i-1} \frac{f_n}{|\Omega|},
\end{align*}
where

\begin{align*}
\eta &= \sum_{h\in\Omega}  \pr(\M_{d}^\plus=h)\left(\ind{h\notin\Omega_f}+\left(1-\frac{f_n}{w(h)}\right)\ind{h\in\Omega_f}\right).
\end{align*}
Hence, $\pr(\hat P=P')$ is independent of $P'$. Thus,
\[
\pr(\hat P=P\mid \hat P\neq\emptyset)=\frac{1}{|\Omega_{f}|}.\qed
\]

\section{Proof of Theorem~\ref{thm:P-P}}

\subsection{Concentration of the number of perfect matchings}

\begin{lemma}
    Let $n$ be even. Let $Y$ be the number of perfect matchings in $\P_*(n,d)$. Then,
    \begin{align}
   %     \ex Y &=\\
        \ex Y^2 &= \left(1+\frac{1}{4d^2}+\frac{2}{3d^3}+O(d^{-4}+d^3/n+\sqrt{d/n}\log^3 n)\right)(\ex Y)^2.\label{eq:conc2}
    \end{align}
\end{lemma}

\proof Let $U=[n]$ and $V=\{e_1,\ldots,e_{dn/2}\}$ and let $\bfd=(\underbrace{d,\ldots,d}_{n})$ and
$\bfs=(\underbrace{2,\ldots,2}_{dn/2})$. Then,
$\P_*(n,d)$ is equivalent to $\G({\bfd}, {\bfs})$  by viewing the elements in $U$ as the vertices in $\P_*(n,d)$ and the elements in $V$ as pairs in $\P_*(n,d)$. That is, if $e_j$ in $V$ is adjacent to two vertices $u,v\in U$ then $e_j$ corresponds to a pair in $\P_*(n,d)$ that joins $u$ and $v$.
Hence, a perfect matching in $\P_*(n,d)$ is equivalent to a subgraph $H$ of $\G({\bfd}, {\bfs})$ such that every vertex in $A=[n]$ has degree one, and exactly $n/2$ vertices in $B=[dn/2]$ has degree two, whereas all the other vertices in $B$ have degree zero. 

The number of ways to choose such a subgraph $H$ is
\[
\frac{n!}{(n/2)! 2^{n/2}} (dn/2)_{n/2}.
\]
For each such given $H$, the probability that it is a subgraph of $\G({\bfd}, {\bfs})$ is $N(\bfs',\bft')/N(\bfs,\bft)$, where
\[
\bfd'=(\underbrace{d-1,\ldots,d-1}_{n}),\quad \bfs'=(\underbrace{2,\ldots,2}_{dn/2-n/2}).
\]
By Theorem~\ref{thm:bipartiteenum}, this probability is
\[
\frac{((d-1)n)! d!^n 2^{dn/2} }{(d-1)!^n2^{(d-1)n/2} (dn)!} \exp\left(-\frac{(d-1)(d-2)n\cdot 2(d-1)n/2}{2(d-1)^2n^2} + \frac{d(d-1)n \cdot 2 dn/2}{2d^2n^2}+O(d^3/n)\right).
\]
Multiplying the above by $\frac{n!}{(n/2)! 2^{n/2}} (dn/2)_{n/2}$ and simplifying by~\eqref{eq:Stirling}, we obtain that
\begin{align*}
    \ex Y=\sqrt{2}(d-1)^{(d-1)n/2} d^{n-dn/2} \exp(1/2+O(d^3/n)).
\end{align*}
Note that this agrees with~\cite[eq.~(2)]{gao2023number} except that the exponent was $1/4$ in~\cite[eq.~(2)]{gao2023number} instead of $1/2$ above. It is then convenient to define $\hat Y$ to be the number of perfect matchings in $\G(n,d)$, and it follows that 
\begin{equation}
\ex Y=\ex \hat Y \cdot \exp(1/4+O(d^3/n)). \label{eq:compare} 
\end{equation}

Next, consider $\ex Y^2$. Fix two perfect matchings $H_1$ and $H_2$ such that $|H_1\cap H_2|=k$. 
\begin{claim}\label{claim:pairs}
    Let $0\le k\le n/2-2$ be an integer. The number of pairs perfect matchings $(H_1,H_2)$ such that $|H_1\cap H_2|=k$ is
    \[
(1+O(n-2k)^{-1})\frac{n!}{2^k k!\sqrt{\pi(n-2k)/2}}.    
    \]
\end{claim}

{\em Proof of Claim~\ref{claim:pairs}.\ } This follows almost exactly the proof of~\cite[Lemma 17]{gao2023number}, with the only difference being the generating function for $(H_1,H_2)$ being
\[
F(z)=\sum_{n=1}^{\infty} \frac{(2n)!}{2\cdot 2n}\cdot 2\cdot \frac{z^{2n}}{(2n)!} = -\frac{1}{2}\log(1-z^2),
\]
instead of $-\frac{1}{2}\log(1-z^2)-z^2/2$ in~\cite[Lemma 17]{gao2023number}. This is due to the fact that there can be edges $e\in H_1$ and $e'\in H_2$ such that $e\neq e'$ but $e,e'$ have the same end vertices, in the pairing model $\P_*(n,d)$. \qed

It is helpful to observe that the number in Claim~\ref{claim:pairs} is asyptotically equal to $\sqrt{e}$ times the number in~\cite[Lemma 17]{gao2023number}.

Fix $(H_1,H_2)$. Then there are $(dn/2)_{n-k}$ ways to represent $H_1\cup H_2$ as a subgraph of $\G({\bfd}, {\bfs})$. Moreover, by Theorem~\ref{thm:bipartiteenum}, each of them appear in $\G({\bfd}, {\bfs})$ with probability
\begin{align*}
&\frac{N(\bfd',\bfs')}{N(\bfd,\bfs)}=
\frac{((d-2)n+2k)! d!^n 2^{dn/2}}{(d-1)!^{2k}(d-2)!^{n-2k}2^{dn/2-n+k} (dn)!}\exp\left(\phi(d,\alpha)+O(d^3/n)\right)
\end{align*}
where
\[
\bfd'=(\underbrace{d-1,\ldots,d-1}_{2k},\underbrace{d-2,\ldots,d-2}_{n-2k}),\quad \bfs'=(\underbrace{2,\ldots,2}_{dn/2-n+k}).
\]
and
\begin{align*}
\phi(d,\alpha)&=-\frac{(2k(d-1)(d-2)+(n-2k)(d-2)(d-3))\cdot 2(dn/2-n+k)}{2((d-2)n+2k)^2}+ \frac{d(d-1)n \cdot 2 dn/2}{2d^2n^2}\\
&=\frac{(-d+3)\alpha+2d-4}{2(d-2+\alpha)}.
\end{align*}
Thus,
\begin{align*}
  \ex Y^2 =&\ex Y+ \sum_{k=0}^{n/2-1}(1+O(n-2k)^{-1})\frac{n!}{2^k k!\sqrt{\pi(n-2k)/2}}(dn/2)_{n-k}   \frac{((d-2)n+2k)! d!^n 2^{dn/2}}{(d-1)!^{2k}(d-2)!^{n-2k}2^{dn/2-n+k} (dn)!}\\
&\times \exp\left(\phi(d,\alpha)+O(d^3/n)\right).
\end{align*}
Note also that this agrees exactly with~\cite[eq.~(16)]{gao2023number} for $\ex \hat Y^2$, except for an additional factor $\sqrt{e}$, and that the value of $\phi$ in~\cite{gao2023number} was
\[
\frac{4(d-2)^2-(d^2-5)\alpha^2-(2d^2-14d+20)\alpha}{4(d-2+\alpha)^2}
\]
for $\ex \hat Y^2$. Hence, following exactly the same argument in~\cite{gao2023number} for $\ex\hat Y^2$, we obtain that
\begin{align*}
\ex Y^2= &\ex\hat Y^2 \cdot\exp\Big(\frac{1}{2}+\frac{(-d+3)\bar\alpha+2d-4}{2(d-2+\bar\alpha)}-\frac{4(d-2)^2-(d^2-5)\bar\alpha^2-(2d^2-14d+20)\bar\alpha}{4(d-2+\bar\alpha)^2}\\
&+O(\sqrt{d/n}\log^3 n+d^3/n)
\Big),
\end{align*}
where $\bar\alpha=1/d$. Substituting $\bar\alpha$ above by $1/d$ and expand around $1/d$,
\begin{align*}
\ex Y^2= &\ex \hat Y^2 \cdot\exp\left(\frac{1}{2}+\frac{2d-3}{2d-2}-\frac{4d^2-10d+5}{4(d-1)^2}
+O(\sqrt{d/n}\log^3 n+d^3/n)
\right)\\
 =&  \ex \hat Y^2 \cdot\exp\left(\frac{1}{2}+\frac{1}{4d^2}+\frac{1}{2d^3}+
+O(d^{-4}+\sqrt{d/n}\log^3 n+d^3/n)
\right)
\end{align*}
We know from~\cite[Theorem 10]{gao2023number} that
\[
\ex \hat Y^2 = \left(1+\frac{1}{6d^3}+O(d^{-4}+d^3/n+\sqrt{d/n}\log^3 n)\right) (\ex\hat Y)^2.
\]
Hence, by~\eqref{eq:compare},
\begin{align*}
\ex Y^2&=\left(\frac{3}{2}+\frac{1}{4d^2}+\frac{2}{3d^3}+O(d^{-4}+d^3/n+\sqrt{d/n}\log^3 n)\right) (\ex\hat Y)^2\\
&=\left(1+\frac{1}{4d^2}+\frac{2}{3d^3}+O(d^{-4}+d^3/n+\sqrt{d/n}\log^3 n)\right) (\ex Y)^2.\qed
\end{align*}

Let $X$ denote the number of double edges in $\P_*(n,d)$ and let $W$ denote the number of triangles in $\P_*(n,d)$. 

\begin{lemma}\label{lem:covariance}
    \begin{align*}
        \ex X &= \frac{(d-1)^2}{4}(1+O(d^3/n)),\quad \var X=\frac{(d-1)^2}{4}(1+O(d^3/n)) \\ 
        \ex W &= \frac{(d-1)^3}{6}(1+O(d^3/n)),\quad \var W=\frac{(d-1)^3}{6}(1+O(d^3/n))\\
        \Cov(X,W) &= O(d^6/n)\\
        \Cov(X, Y) &= \left(1/d^2+2/d^3+O(1/d^4+d/n)\right)\ex X\ex Y\\
        \Cov(W, Y) &= \left(-1/d^3+O(1/d^4+d/n)\right)\ex W\ex Y.
    \end{align*}
\end{lemma}

\proof We apply the bipartite version of~\cite[Theorem 1]{gao2023subgraph}.
\begin{align*}
\ex X=& \binom{n}{2} \binom{(dn/2)}{2} \frac{4d^2(d-1)^2}{(dn)^4} (1+O(d^3/n))=\frac{(d-1)^2}{4}(1+O(d^3/n))\\
\ex W=& \binom{n}{3} (dn/2)_3 \frac{2^3 d^3(d-1)^3}{(dn)^6}(1+O(d^3/n))=\frac{(d-1)^3}{6}(1+O(d^3/n))\\
\Cov(X,W) =& \ex XW-\ex X\ex W = O(d^6/n).
\end{align*}

\begin{align*}
\Cov(X,Y) =& \ex Y\left(\frac{n^2}{4}\frac{(d-1)^2_2}{(dn-n)^2}+\frac{n}{2} \frac{(d-1)^2}{dn-n}\right)(1+O(d^3/n))-\ex X\ex Y\\
=&\left(1/d^2+2/d^3+O(1/d^4+d/n)\right)\ex X\ex Y\\ 
\Cov(W,Y) =& \left(-1/d^3+O(1/d^4+d/n)\right)\ex W\ex Y
\end{align*}

\qed

Project $Y$ onto $X$ and $W$ as follows. Let 
\begin{align*}
Y^* &=aX+bW+c\\
a &=\Cov(X,Y)/\var X,\quad b=\Cov(W,Y)/\var W,\quad c=\ex Y-a\ex X-b\ex W.
\end{align*}

\begin{lemma}\label{lem:Ystar}
    $
\var(Y-Y^*)=O(\xi)(\ex Y)^2,    
    $ where $\xi=d^{-4}+d^3/n+\sqrt{d/n}\log^3n$.
\end{lemma}

\proof
By definition of $a,b, c$, $\Cov(Y-Y^*,Y^*)=-2ab\Cov(X,W)=O(d/n)(\ex Y)^2$.
Thus, 
\begin{align*}
\var(Y-Y^*) &= \Cov(Y,Y-Y^*)-\Cov(Y^*,Y-Y^*)\\
&= \var(Y)-a\Cov(Y,X)-b\Cov(Y,W)-\Cov(Y^*,Y-Y^*) \\
&= \left(\frac{1}{4d^2}+\frac{2}{3d^3}- \left(1/d^2+2/d^3\right)^2\ex X- \left(-1/d^3\right)^2\ex W+O(\xi)\right)(\ex Y)^2=O(\xi)(\ex Y)^2.\qed
\end{align*}

Using a standard switching argument we obtain the following concentration bound on $X$ and $W$. We skip the proof. %The proof is included in the Appendix.
\begin{lemma}~\label{lem:concentrationXW}
    Let $C>0$ be sufficiently large. 
    \begin{enumerate}[(a)]
        \item $\pr(|X-\ex X|>C\sqrt{\log d\ex X}) < 1/d^2$;
        \item $\pr(|W-\ex W|>C\sqrt{\log d\ex W}) < 1/d^2$
    \end{enumerate}
\end{lemma}

%\begin{theorem}
 %   Let $d=\omega(1)$. 
  %  \begin{enumerate}[(a)]
   %     \item There is a coupling $(P,M,P')$ such that marginally $P\sim \P_*(n,d)$, $M$ is a uniform random matching on $[n]$, $P'\sim \P_*(n,d+1)$ and jointly $\pr(P+M=P')=1-O(\sqrt{\log d}/d)$.
    %    \item There is a coupling $(P,M,P')$ such that marginally $P\sim \P_*(n,d)$, $M\sim M_2^\plus$, $P'\sim \P_*(n,d+1)$ and jointly $\pr(P'\subseteq P+M)= 1-O(\log d/d^{1.5})$.
 %   \end{enumerate}
%\end{theorem}

\noindent {\em Proof of Theorem~\ref{thm:P-P}.\ } 
We first show that 
\begin{align}
|Y-\ex Y|=O(\sqrt{\log d}/d)\ex Y,\quad \text{with probability at least $1-O(d^{-2})$}.\label{eq:conc}     
\end{align}

Let $C>0$ be a sufficiently large constant.

\begin{align*}
    \pr(|Y-\ex Y|\ge (2C\sqrt{\log d}/d) \ex Y) \le &    \pr(|Y-Y^*|\ge (C\sqrt{\log d}/d) \ex Y) \\
    &+     \pr(|Y^*-\ex Y^*|\ge (C\sqrt{\log d}/d) \ex Y^*). 
\end{align*}
By Lemma~\ref{lem:Ystar}, $$ \pr(|Y-Y^*|\ge (C\sqrt{\log d}/d) \ex Y)=O\left((d^{2}/\log d)\frac{\var (Y-Y^*)}{(\ex Y)^2}\right) =O(d^{-2}/\log d). $$
By Lemma~\ref{lem:concentrationXW} (with sufficiently large $C$), with probability at least $d^{-2}$,
\[
|Y^*-\ex Y^*| =O(a\sqrt{\log d\cdot \ex X}+b\sqrt{\log d \cdot\ex W}) = O(\sqrt{\log d}/d) \ex Y.
\]
We have established~\eqref{eq:conc}.

Let $\Gamma$ be the set of perfect matchings on $[n]$, and let $D$ be a bipartite graph on $S\cup T$ where $S=\Omega(\P_*(n,d))\times \Gamma$ and $T=\Omega(\P_*(n,d+1))$, and $(P,M)\in S$ is adjacent to $P'\in T$ if $P+M=P'$. Let $\eps=C\sqrt{\log d}/d$ for some large $C>0$. Define $S_{\text{good}}$ and $T_{\text{good}}$ as in 
Corollary~\ref{cor:coupling2}.
Then, every element in $S$ has degree one, and thus $|S_{\text{good}}|=|S|$. On the other hand, for each $P'\in T$, $\text{deg}_D(P')=Y(P')$. The desired coupling now follows by~\eqref{eq:conc} and Corollary~\ref{cor:coupling2}. \qed

%\begin{theorem} (Guessed theorem)
 %   Let $i,j=O(d^2)$, $i\le j$. Perhaps for typical $i,j$ meaning deviation is $O(d\sqrt{\log d})$.  Then, there is a coupling that,
  %  with probability $1-d^{-1.1}$, $\P^{(i)}(n,d)+M=\P^{(j)}(n,d+1)$. \jc{But the distribution of $M$ depends on $i$ and $j$ and thus $M$ is not independent of $\P^{(i)}(n,d)$.}

%\jc{Non-ideal fix: use a set of independent matchings until finding the right one; i.e.\ increasing the number of double edges by $j-i$. If $i$ and $j$ are both of typical values then $j-i$ should have value around $d$. The probability of hitting the exact value should be around $1/\sqrt{d}$ by local limit. Hence, we need around $\sqrt{d}$ matchings. That gives $\sum_{i\le d} \sqrt{i} \approx d^{3/2}$ in the final coupling.}
    
%\end{theorem}

\section{Proof of Theorem~\ref{thm:P-G}}

%Let $\P_*(n,d)$ denote the loopless pairing model, and $\M(n,d)$ the multigraph obtained from $\P_*(n,d)$. 

%Suppose that $G$ is a simple graph on $[n]$ and $M$ is a multigraph on $[n]$. We say $G\subseteq M$ if for every $\{u,v\}\in [n]^2$, if $uv$ is an edge in $G$ then $u$ and $v$ are adjacent in $M$, either by a simple edge, or a multiple edge. 

%\begin{theorem}\label{thm:graph-pairing}
 %   Suppose that $d\to \infty$ and $d\ll n^{1/4}$. There exists a coupling $(G_*,M^*)$ such that $G_*\sim \G(n,d)$, $M^*\sim \M(n,d+1)$, and a.a.s.\ $G_*\subseteq M^*$.
%\end{theorem}

The proof strategy is similar to that of Theorem~\ref{thm:P-P}. But instead of considering the number of perfect matchings in $\P_*(n,d)$, we  condition on the number of double edges, and considering perfect matchings that ``cover'' all the double edges.

Let $\P^{(i)}(n,d)$ denote the set of pairings in $\P_*(n,d)$ containing exactly $i$ pairwise vertex-disjoint double edges, and no other types of multiple edges. % Let $\M^{(i)}(n,d)$ be the multigraph corresponding to $\P^{(i)}(n,d)$.

Given a pairing $P\in \P^{(i)}(n,d)$, a perfect matching in $P$ covering all double edges of $P$ is a set $M$ of pairs such that $M$ corresponds to a perfect matching in $G[P]$, and for every double edge $uv$ in $G[P]$, there is a pair in $M$ that joins a point in $u$ and a point in $v$.

Let $Y=Y_n^{(i)}$ be the random variable denoting the number of perfect matchings of $P\sim\P^{(i)}(n,d)$ which covers every double edge of $P$.

\subsection{Concentration of $Y_n^{(i)}$ in $\P^{(i)}(n,d)$}

\begin{lemma}\label{lem:YPairing} Let $i=O(d^2)$. Then,
%\[
%\ex Y=
%\]    
%and
\begin{equation}
\ex Y^2=\left(1+\frac{1}{6d^3}+O\left(d^{-4}+\frac{d^3}{n}+\sqrt{\frac{d}{n}}\log^3 n\right)\right)(\ex Y)^2.\label{eq:conc1}
\end{equation}

\end{lemma}

{\bf Remark}. Comparing~\eqref{eq:conc2} with~\eqref{eq:conc1}, we see that the variance in~\eqref{eq:conc2} is greater. This is caused by the variance of the number of double edges in $\P_*(n,d)$. This difference in variance directly causes the difference in the probability bounds in Theorem~\ref{thm:P-P} and in Theorem~\ref{thm:P-G}.
\smallskip

\noindent {\em Proof of Lemma~\ref{lem:YPairing}.\ } Consider the probability space $\P'$ of $\P^{(i)}(n,d)$ conditioned to the event that the $i$ vertex-disjoint double edges are precisely $\{\{2j-1,2j\}, 1\le j\le i\}$. By symmetry, it suffices to establish the assertion of the lemma in this probability space. Moreover, letting $S=[n]\setminus [2i]$, each perfect matching on $S$ corresponds to exactly $2^i$ perfect matchings in $P\in\P'$ that cover all the double edges in $P$. Hence, it is convenient to define $\Z$ to be the set of matchings in $P$ that saturates all and only the vertices in $S$, and it follows that $Y=2^i Z$ where $Z=|\Z|$. 

Let $X=\{\{2j-1,2j\}, 1\le j\le i\}$.
Let $\G$ be the set of simple graphs on $[n]$ with degree sequence $(\underbrace{d-2,\ldots, d-2}_{2i}, \underbrace{d,\ldots, d}_{n-2i})$ avoiding all the edges in $X$. Then $\G$ corresponds precisely to the set of $d$-regular multi-graphs on $[n]$ in $\P'$. 

As in Theorem~\ref{thm:subgraphenum}, given any simple graph $W$ on $[n]$, let $N(\bfg,W)$ be the number of simple graphs $G$ on $[n]$ whose degree sequence is $\bfg$ and $G\cap W=\emptyset$. 

Let $H$ be a fixed perfect matching on $S$. Then,
\[
\pr(H\subseteq \P')=\frac{N(\bfg, X\cup H)}{N(\bfg',X)},
\]
where
\[
\bfg=(\underbrace{d-2,\ldots,d-2}_{2i},\underbrace{d-1,\ldots,d-1}_{n-2i}),\quad \bfg'=(\underbrace{d-2,\ldots,d-2}_{2i},\underbrace{d,\ldots,d}_{n-2i})
\]
We apply Theorem~\ref{thm:subgraphenum} to estimate $N(\bfg, X\cup H)$ and $N(\bfg',X)$. It is immediate that
\begin{align*}
    m(\bfg)=\frac{n(d-1)}{2}-i, \quad m(\bfg')=\frac{nd}{2}-2i,
\end{align*}
and by the assumption that $i=O(d^2)$,
\begin{align*}
    \lambda(\bfg)=\frac{\sum_{j=1}^n g_j(g_j-1)}{4m(\bfg)}=\frac{1}{2}(d-2)(1+O(d/n)), \quad \lambda(\bfg')=\frac{1}{2}(d-1)(1+O(d/n)),
\end{align*}
and
\begin{align*}
    \mu(\bfg,X\cup H)=\frac{\sum_{jk\in X\cup H} g_jg_k}{2m(\bfg)}=\frac{1}{2}(d-1)(1+O(d/n)), \quad \mu(\bfg',X)=O(d^3/n).
\end{align*}
It follows that
\begin{align*}
 N(\bfg, X\cup H) &= \frac{(n(d-1)-2i)!\exp\Big(-\lambda(\bfg)-\lambda(\bfg)^2-\mu(\bfg,X\cup H)+O(d^3/n)\Big)}{(n(d-1)/2-i)!2^{(n(d-1)/2-i)}(d-2)!^{2i}(d-1)!^{n-2i}}, \\
 N(\bfg', X) &= \frac{(nd-4i)!\exp\Big(-\lambda(\bfg')-\lambda(\bfg')^2-\mu(\bfg',X)+O(d^3/n)\Big)}{(nd/2-2i)!2^{(nd/2-2i)}(d-2)!^{2i}d!^{n-2i}}.
\end{align*}
There are $\frac{(n-2i)!}{(n/2-i)!2^{n/2-i}}$ choices for $H$ as a perfect matching over $S$. Thus,
\begin{align*}
\ex Z=&\frac{(n-2i)!}{(n/2-i)!2^{n/2-i}} \frac{(n(d-1)-2i)!(nd/2-2i)!2^{(nd/2-2i)}(d-2)!^{2i}d!^{n-2i}}{(n(d-1)/2-i)!2^{(n(d-1)/2-i)}(d-2)!^{2i}(d-1)!^{n-2i}(nd-4i)!}\\
&\times \exp\Big(-\lambda(\bfg)-\lambda(\bfg)^2-\mu(\bfg,X\cup H)+\lambda(\bfg')+\lambda(\bfg')^2+\mu(\bfg',X)+O(d^3/n)\Big)\\
=&\frac{(n-2i)!(n(d-1)-2i)!(nd/2-2i)!d^{n-2i}}{(n/2-i)!(n(d-1)/2-i)!(nd-4i)!} \exp(1/4+O(d^3/n)).
\end{align*}
By Stirling formula we know that
\begin{equation}
\frac{(2N)!}{N!}=(1+O(N^{-1}))\sqrt{2}\left(\frac{4N}{e}\right)^N, \label{eq:Stirling}
\end{equation}
which yields that 
\begin{align*}
\ex Z &=\exp(1/4+O(d^3/n))\sqrt{2}\left(\frac{2n-4i}{e}\right)^{n/2-i} \left(\frac{2n(d-1)-4i}{e}\right)^{n(d-1)/2-i} \left(\frac{e}{2nd-8i }\right)^{nd/2-2i} d^{n-2i} \\
&= \exp(1/4+O(d^3/n))\sqrt{2} \frac{(n-2i)^{n/2-i}(n(d-1)-2i)^{n(d-1)/2-i}}{(nd-4i)^{nd/2-2i}}d^{n-2i}.
\end{align*}
Using that $(n-2i)^{n/2-i} = n^{n/2-i} \exp(-i+O(d^4/n))$, $(n(d-1)-2i)^{n(d-1)/2-i}=(n(d-1))^{n(d-1)/2-i}\exp(-i+O(d^4/n))$ and $(nd-4i)^{nd/2-2i}=(nd)^{nd/2-2i}\exp(-2i+O(d^4/n))$, we further simplify the above to obtain that
\[
\ex Z =\exp(1/4+O(d^4/n))\sqrt{2} (d-1)^{n(d-1)/2-i} d^{-n(d-2)/2}.
\]

The computation for $\ex Z^2$ is similar to~\cite{gao2023number}. Fix two perfect matchings $H_1$ and $H_2$ over $S$ such that $|H_1\cap H_2|=k$. By~\cite[Lemma 17]{gao2023number}, the number of choices for $(H_1,H_2)$ is
\[
(1+O(n-2i-2k)^{-1})\frac{(n-2i)!}{2^k k!\sqrt{e\pi(n-2i-2k)/2}}.
\]

Moreover,
\[
\pr(H_1\cup H_2\subseteq \P')=\frac{N(\bfg, X\cup H_1\cup H_2)}{N(\bfg',X)},
\]
where, without loss of generality,
\[
\bfg=(\underbrace{d-2,\ldots,d-2}_{2i}, \underbrace{d-1,\ldots,d-1}_{2k}, \underbrace{d-2,\ldots,d-2}_{n-2i-2k}),\quad \bfg'=(\underbrace{d-2,\ldots,d-2}_{2i},\underbrace{d,\ldots,d}_{n-2i}).
\]
It is immediate that
\begin{align*}
    m(\bfg)=\frac{n(d-2)+2k}{2},
\end{align*}
and assuming that $i=O(d^2)$ and letting $\alpha=\alpha(k):=2k/n$,
\begin{align*}
    \lambda(\bfg)=\frac{(d-2)(d-3)(1-\alpha)+(d-1)(d-2)\alpha}{2(d-2)+2\alpha},
\end{align*}
and
\begin{align*}
    \mu(\bfg,X\cup H_1\cup H_2)=\frac{(1-\alpha)(d-2)^2+(d-1)^2\alpha/2}{d-2+\alpha}+O(d^3/n).
\end{align*}
Hence,
\[
N(\bfg,X\cup H_1\cup H_2) = \frac{(n(d-2)+2k)!\exp\Big(-\lambda(\bfg)-\lambda(\bfg)^2-\mu(\bfg,X\cup H_1\cup H_2)\Big)}{(n(d-2)/2+k)!2^{n(d-2)/2+k}(d-2)!^{n-2k}(d-1)!^{2k}}.
\]
It follows that $\pr(H_1\cup H_2\subseteq \P')=\rho_2(n,d,\alpha(k))$ where
\begin{align*}
\rho_2(n,d,\alpha(k))    =&\frac{(n(d-2)+2k)!(nd/2-2i)!2^{(nd/2-2i)}(d-2)!^{2i}d!^{n-2i}}{(n(d-2)/2+k)!2^{n(d-2)/2+k}(d-2)!^{n-2k}(d-1)!^{2k} (nd-4i)!}\\
    &\times \exp\left(\frac{(-d^2+5)\alpha^2+(-2d^2+14d-20)\alpha+4(d-2)^2}{4(d-2+\alpha)^2}\right)\\
    =&\frac{(n(d-2)+2k)!(nd/2-2i)!}{(n(d-2)/2+k)! (nd-4i)!} 2^{n-2i-k} d^{n-2i}(d-1)^{n-2i-2k} \\
    &\times \exp\left(\frac{(-d^2+5)\alpha^2+(-2d^2+14d-20)\alpha+4(d-2)^2}{4(d-2+\alpha)^2} +O(d^3/n)\right)\\
    =&  \left(\frac{2n(d-2)+4k}{e}\right)^{n(d-2)/2+k} \left(\frac{e}{2nd-8i}\right)^{nd/2-2i} 2^{n-2i-k} d^{n-2i}(d-1)^{n-2i-2k} \\
    &\times \exp\left(\frac{(-d^2+5)\alpha^2+(-2d^2+14d-20)\alpha+4(d-2)^2}{4(d-2+\alpha)^2} +O(d^3/n)\right)
\end{align*}

Therefore,

\begin{align*}
\ex Z^2 &=\ex Z+\sum_{k=0}^{n/2-i-1}(1+O(n-2i-2k)^{-1})\frac{(n-2i)!}{2^k k!\sqrt{e\pi(n-2i-2k)/2}} \rho_2(n,d,\alpha(k))\\
&=\ex Z+\sum_{k=0}^{n/2-i-1}(1+O(n-2i-2k)^{-1}) \sqrt{\frac{2}{e\pi}}(n-2i)!\varphi(k),
\end{align*}
where
\[
\varphi(k)=\frac{\rho_2(n,d,\alpha(k))}{2^k k!\sqrt{n-2i-2k}}.
\]
Following exactly the same approach as in~\cite[Lemma 18]{gao2023number}, we found that
\begin{equation}
\ex Z^2=\left(1+O\left(\sqrt{\frac{d}{n}}\log^3 n+\frac{d^3}{n}\right)\right) \sqrt{\frac{2\pi}{\bar\delta}} \frac{(n-2i)!}{2^{\bar k} \bar k!\sqrt{e\pi(n-2i-2\bar k)/2}} \rho_2(n,d,\alpha(\bar k)),\label{eq:2nd}
\end{equation}
where $\bar \alpha=1/d$, $\bar k=\lfloor\bar\alpha n/2\rfloor$, and $\bar\delta=\frac{2d}{n}\frac{d(d-2)}{(d-1)^2}$. We briefly explain why all the calculations in~\cite[Lemma 18]{gao2023number} are derived here in exactly the same way, with exactly the same $\bar\alpha,\bar k$ and $\bar \delta$. 
The key part of this proof is estimating the ratio $\varphi(k)/\varphi(k-1)$, determining the value of $\bar k$
around which the ratio is approximately 1, and estimating the sum above over $k$ by an integral. By the definition of $\varphi$ above, it is straightforward to see that $\varphi(k)/\varphi(k-1)$ agrees exactly with~\cite[eq.~(20)]{gao2023number}, which then leads to~\eqref{eq:2nd} following exactly the same approach in~\cite{gao2023number}.
Substituting the values of $\bar\delta$ and $\bar k$ we obtain that
\begin{align*}
    \ex Z^2 = & %\sqrt{\frac{2\pi}{\bar\delta}} \frac{(n-2i)!}{2^{\bar k} \bar k!\sqrt{e\pi(n-2i-2\bar k)/2}}  \left(\frac{2n(d-2)+4k}{e}\right)^{n(d-2)/2+k} \left(\frac{e}{2nd-8i}\right)^{nd/2-2i}  \\
    \sqrt{\frac{2(d-1)}{ed(d-2)}} \frac{(n-2i)!}{2^{n/2d} (n/2d)!}\left(\frac{2n(d-2)+2n/d}{e}\right)^{n(d-2)/2+n/2d} \left(\frac{e}{2nd-8i}\right)^{nd/2-2i} \\
    &\times 2^{n-2i-n/2d} d^{n-2i}(d-1)^{n-2i-n/d}\cdot \exp\left(\frac{4d^2-10d+5}{4(d-1)^2} +O(\sqrt{d/n}\log^3 n+ d^3/n)\right).
\end{align*}
Applying Stirling's formula, we obtain
\begin{align*}
    \ex Z^2 = & %\sqrt{\frac{2\pi}{\bar\delta}} \frac{(n-2i)!}{2^{\bar k} \bar k!\sqrt{e\pi(n-2i-2\bar k)/2}}  \left(\frac{2n(d-2)+4k}{e}\right)^{n(d-2)/2+k} \left(\frac{e}{2nd-8i}\right)^{nd/2-2i}  \\
    2\sqrt{\frac{(d-1)}{e(d-2)}} \frac{((n-2i)/e)^{n-2i}}{2^{n/2d} (n/2ed)^{n/2d}}\left(\frac{2n(d-2)+2n/d}{e}\right)^{n(d-2)/2+n/2d} \left(\frac{e}{2nd-8i}\right)^{nd/2-2i}  \\
    &\times 2^{n-2i-n/2d} d^{n-2i}(d-1)^{n-2i-n/d} \cdot\exp\left(\frac{4d^2-10d+5}{4(d-1)^2} +O(\sqrt{d/n}\log^3 n+ d^3/n)\right).
\end{align*}
Again, with $(n-2i)^{n-2i}=n^{n-2i}\exp(-2i+O(d^4/n))$, $(2nd-8i)^{nd/2-2i}=(2nd)^{nd/2-2i}\exp(-2i)$ we can further simplify the above and obtain that
\begin{align*}
    \ex Z^2 = & 2\sqrt{\frac{(d-1)}{e(d-2)}} 2^{-n+n/2d+2i} d^{n/2d+2i-nd/2}(d-2+1/d)^{(n/2)(d-2+1/d)} \\
    &\times 2^{n-2i-n/2d} d^{n-2i}(d-1)^{n-2i-n/d}\cdot\exp\left(\frac{4d^2-10d+5}{4(d-1)^2} +O(\sqrt{d/n}\log^3 n+ d^3/n)\right)\\
    =&  2\sqrt{\frac{(d-1)}{e(d-2)}}  d^{-nd+2n}(d-1)^{nd-n-2i} \\
    &\times   \exp\left(\frac{4d^2-10d+5}{4(d-1)^2} +O(\sqrt{d/n}\log^3 n+ d^3/n)\right).
\end{align*}

Finally, with $\xi=\sqrt{d/n}\log^3 n+ d^4/n$
\begin{align*}
    \frac{\ex Z^2}{(\ex Z)^2} &= (1+O(\xi)) \sqrt{\frac{(d-1)}{e(d-2)}} \exp\left(\frac{4d^2-10d+5}{4(d-1)^2} -1/2 \right)= 1+\frac{1}{6d^3}+O(d^{-4}+\xi),
\end{align*}
which implies the assertion of the lemma by recalling that $Y=2^i Z$. \qed

\subsection{Proof of Theorem~\ref{thm:P-G}}

The following lemma follows by a standard switching argument, which we skip and leave as an easy exercise.

\begin{lemma}\label{lem:TypicalPairing} 
\begin{enumerate}[(a)]
\item For every $i=O(d^2)$,   \[
    \pr_{\P_*(n,d)}\Big(P\in \P^{(i)}(n,d)\mid\ \text{$P$ has exactly $i$ double edges}\Big)=1-O(d^4/n).
    \]
    \item $\pr_{\P_*(n,d)}\Big(\cup_{i=0}^{d^2} \P^{(i)}(n,d)\Big)=1-O(d^4/n)$.    
\end{enumerate}
\end{lemma}

\begin{proof}[Proof of Theorem~\ref{thm:P-G}]

We apply~\cite[Corollary 2.2]{isaev2025sprinkling}. Let $D$ be the bipartite graph on $(\G,\M)$ where $\G$ is the set of graphs in the probability  space of $\G(n,d)$ and $\M$ is the set of $(d+1)$-regular multigraphs on $[n]$ containing exactly $i$ vertex-disjoint double edges and no any other types of multiple edges. A graph $G\in \G$ is adjacent to a multigraph $M\in \M$ in $D$ if $M\setminus G$ is a perfect matching on $[n]$ that covers every double edge of $M$. Let $\eps=C d^4/n +1/d$ for some sufficiently large constant $C>0$. Let $S_{\text{good}}$ and $T_{\text{good}}$ be defined as in Corollary~\ref{cor:coupling2}. 

By Lemmas~\ref{lem:YPairing} and~\ref{lem:TypicalPairing} and the same treatment as in Lemmas~\ref{lem:covariance}--\ref{lem:concentrationXW} and in the proof of Theorem~\ref{thm:P-P}, 
\[
|S_{\text{good}}|\ge (1-\delta)|S|,\quad |T_{\text{good}}|\ge (1-\delta)|T|
\]
    with $\delta=c\sqrt{\log d}/d^{1.5}$ for some constant $c>0$. 
By Corollary~\ref{cor:coupling2}, there is a coupling $\pi_i$ of $(G,M)$ where $G$ is uniform in $\G$ and $M$ is uniform in $\M$ and $\pr(G\subseteq M)=1-O(\sqrt{\log d}/d^{1.5}+d^4/n)$.

Finally, construct a coupling $(G,M)$ where $G\sim \G(n,d)$ and $M=G[P]$ where $P\in \P_*(n,d+1)$ as follows. First, sample $P\in \P_*(n,d+1)$. If $P\notin \cup_{i=0}^{(d+1)^2} \P^{(i)}(n,d+1)$, sample $G~\sim \G(n,d)$ independently. Otherwise, construct $G$ according to $\pi_i$ given $M=G[P]$ where $i$ denotes the number of double edges in $M$. Obviously, $G$ and $M$ have their desired marginal distribution. Moreover, the probability that $G\nsubseteq M$ is $O(\sqrt{\log d}/d^{1.5}+d^4/n)$.
\end{proof}

\section{Proofs of Theorems~\ref{thm:P-P2} and~\ref{thm:P-M} and Corollary~\ref{thm:G-M}}

By the remark after Lemma~\ref{lem:YPairing}, in order to improve the probability bound in Theorem~\ref{thm:P-P}, it is necessary to condition on the number of double edges in $\P_*(n,d)$. Thus, we consider $\P^{(i)}(n,d)$ again. In this section, we show concentration of the number of perfect matchings, and the number of perfect matchings that cover a specific number of double edges.

\subsection{Concentration of the number of perfect matchings}

Using basically the same but more delicate arguments as in the proof of Lemma~\ref{lem:YPairing}, we may show the following.

\begin{lemma}\label{lem2:YPairing} Let $i=O(d^2)$, $0\le a\le i$. Let $Y$ be the number of perfect matchings of $\P^{(i)}(n,d)$, and let Let $Y_a$ be the number of perfect matchings $M$ of $P\sim\P^{(i)}(n,d)$ such that $P-M$ has $i-a$ double edges.
Then,
%\[
%\ex Y=
%\]    
%and
\[
\ex Y^2=\left(1+\frac{1}{6d^3}+O\left(d^{-4}+\frac{d^3}{n}+\sqrt{\frac{d}{n}}\log^3 n\right)\right)(\ex Y)^2,
\]
and
\[
\ex Y_a^2=\left(1+\frac{1}{6d^3}+O\left(d^{-4}+\frac{d^3}{n}+\sqrt{\frac{d}{n}}\log^3 n\right)\right)(\ex Y_a)^2,
\]

\end{lemma}

\proof Consider the probability space $\P'$ as in the proof of Lemma~\ref{lem:YPairing}. Let $0\le j\le i$, and let $Z_j$ be the number of perfect matchings that uses exactly $j$ pairs contained in the $i$ double edges. There are $\binom{i}{j} 2^j$ ways to fix these $j$ pairs. Let $S$ be the set of remaining $n-2j$ vertices. 
\begin{claim}
    The number of perfect matchings on $S$ avoiding any pairs contained in a double edge is
    \[
    (1+O(d^2/n)) \frac{(n-2j)!}{(n/2-j)! 2^{n/2-j}}.
    \]
\end{claim}
Hence, the total number of perfect matchings on $[n]$ that uses exactly $j$ pairs in the double edges is
\begin{equation}
(1+O(d^2/n)) \frac{(n-2j)!}{(n/2-j)! 2^{n/2-j}}\binom{i}{j} 2^j. \label{eq:Hnumber}
\end{equation}
Fix such a perfect matching $H$. Without loss of generality, we may assume that the $j$ double edges covered by $H$ are $\{\{2h-1,2h\},1\le h\le j\}$ and thus $S=[n]\setminus [2j]$. 

As before, let $X=\{\{2j-1,2j\}, 1\le j\le i\}$.
Let $\G$ be the set of simple graphs on $[n]$ with degree sequence $(\underbrace{d-2,\ldots, d-2}_{2i}, \underbrace{d,\ldots, d}_{n-2i})$ avoiding all the edges in $X$. Then $\G$ corresponds precisely to the set of $d$-regular multi-graphs on $[n]$ in $\P'$. 

As in Theorem~\ref{thm:subgraphenum}, given any simple graph $W$ on $[n]$, let $N(\bfg,W)$ be the number of simple graphs $G$ on $[n]$ whose degree sequence is $\bfg$ and $G\cap W=\emptyset$. 
 Then,
\[
\pr(H\subseteq \P')=\frac{N(\bfg, X\cup H)}{N(\bfg',X)},
\]
where
\[
\bfg=(\underbrace{d-2,\ldots,d-2}_{2j},\underbrace{d-3,\ldots,d-3}_{2i-2j},\underbrace{d-1,\ldots,d-1}_{n-2i}),\quad \bfg'=(\underbrace{d-2,\ldots,d-2}_{2i},\underbrace{d,\ldots,d}_{n-2i})
\]
We apply Theorem~\ref{thm:subgraphenum} to estimate $N(\bfg, X\cup H)$ and $N(\bfg',X)$. It is immediate that
\begin{align*}
    m(\bfg)=\frac{n(d-1)}{2}-2i+j, \quad m(\bfg')=\frac{nd}{2}-2i,
\end{align*}
and by the assumption that $i=O(d^2)$,
\begin{align*}
    \lambda(\bfg)=\frac{\sum_{j=1}^n g_j(g_j-1)}{4m(\bfg)}=\frac{1}{2}(d-2)(1+O(d/n)), \quad \lambda(\bfg')=\frac{1}{2}(d-1)(1+O(d/n)),
\end{align*}
and
\begin{align*}
    \mu(\bfg,X\cup H)=\frac{\sum_{jk\in X\cup H} g_jg_k}{2m(\bfg)}=\frac{1}{2}(d-1)(1+O(d/n)), \quad \mu(\bfg',X)=O(d^3/n).
\end{align*}
It follows that
\begin{align*}
 N(\bfg, X\cup H) &= \frac{(n(d-1)-4i+2j)!\exp\Big(-\lambda(\bfg)-\lambda(\bfg)^2-\mu(\bfg,X\cup H)+O(d^3/n)\Big)}{(n(d-1)/2-2i+j)!2^{(n(d-1)/2-2i+j)}(d-2)!^{2j}(d-3)!^{2i-2j}(d-1)!^{n-2i}}, \\
 N(\bfg', X) &= \frac{(nd-4i)!\exp\Big(-\lambda(\bfg')-\lambda(\bfg')^2-\mu(\bfg',X)+O(d^3/n)\Big)}{(nd/2-2i)!2^{(nd/2-2i)}(d-2)!^{2i}d!^{n-2i}}.
\end{align*}
Combining~\eqn{eq:Hnumber} we obtain
\begin{align*}
\ex Z_j=&  \frac{(n-2j)!}{(n/2-j)! 2^{n/2-j}} \binom{i}{j} 2^j \\
&\times\frac{(n(d-1)-4i+2j)!(nd/2-2i)!2^{(nd/2-2i)}(d-2)!^{2i}d!^{n-2i}}{(n(d-1)/2-2i+j)!2^{(n(d-1)/2-2i+j)}(d-2)!^{2j}(d-3)!^{2i-2j}(d-1)!^{n-2i}(nd-4i)!}\\
&\times \exp\Big(-\lambda(\bfg)-\lambda(\bfg)^2-\mu(\bfg,X\cup H)+\lambda(\bfg')+\lambda(\bfg')^2+\mu(\bfg',X)+O(d^3/n)\Big)\\
=&\binom{i}{j}\frac{(n-2j)!(n(d-1)-4i+2j)!(nd/2-2i)!d^{n-2i}(d-2)^{2i-2j} 2^j}{(n/2-j)!(n(d-1)/2-2i+j)!(nd-4i)!} \exp(1/4+O(d^3/n)).
\end{align*}
By Stirling formula we know that
\begin{equation}
\frac{(2N)!}{N!}=(1+O(N^{-1}))\sqrt{2}\left(\frac{4N}{e}\right)^N, \label{eq:Stirling}
\end{equation}
which yields that 
\begin{align*}
\ex Z_j &=\exp(1/4+O(d^3/n))\sqrt{2}\left(\frac{2n-4j}{e}\right)^{n/2-j} \left(\frac{2n(d-1)-8i+4j}{e}\right)^{n(d-1)/2-2i+j} \left(\frac{e}{2nd-8i }\right)^{nd/2-2i}\\
&\hspace{1cm}\times \binom{i}{j} d^{n-2i} (d-2)^{2i-2j} 2^j \\
&= \exp(1/4+O(d^3/n))\sqrt{2} \frac{(n-2j)^{n/2-j}(n(d-1)-4i+2j)^{n(d-1)/2-2i+j}}{(nd-4i)^{nd/2-2i}} \binom{i}{j} d^{n-2i} (d-2)^{2i-2j} 2^j.
\end{align*}
Using that $(n-2j)^{n/2-j} = n^{n/2-j} \exp(-j+O(d^4/n))$, $(n(d-1)-4i+2j)^{n(d-1)/2-2i+j}=(n(d-1))^{n(d-1)/2-2i+j}\exp(-2i+j+O(d^4/n))$ and $(nd-4i)^{nd/2-2i}=(nd)^{nd/2-2i}\exp(-2i+O(d^4/n))$, we further simplify the above to obtain that
\[
\ex Z_j =\exp(1/4+O(d^4/n))\sqrt{2} \binom{i}{j} 2^j (d-1)^{n(d-1)/2-2i+j} d^{-n(d-2)/2} (d-2)^{2i-2j}.
\]
Summing over $0\le j\le i$ we obtain that for some function $g(n,d)$ which is independent of $i$,
\begin{eqnarray}
\ex Y &= & \sum_{0\le j\le i} \ex Z_j = \exp(1/4+O(d^4/n))\sqrt{2} (d-1)^{n(d-1)/2-2i}  d^{-n(d-2)/2} (d-2)^{2i} \left(1+\frac{2(d-1)}{(d-2)^2}\right)^i \nonumber\\
&=& (1+O(d^4/n)) g(n,d) \left(\frac{d^2-2d+2}{(d-1)^2}\right)^i \label{eq:FM}
\end{eqnarray}

The computation for $\ex Y^2$ is similar to~\cite{gao2023number}. 

Let $j,k,h,k,\eta$ be nonnegative integers such that $j+\ell+h +\eta/2\le i$. Fix two perfect matchings $H_1$ and $H_2$ over $[n]$ such that 
\begin{enumerate}[(P1)]
    \item there are exactly $j$ double edges covered by both $H_1$ and $H_2$;
    \item there are exactly $\ell$ double edges covered by $H_1$ but not $H_2$;
    \item there are exactly $h$ double edges covered by $H_2$ but not $H_1$;
    \item there are exactly $k$ edges in $H_1\cup H_2$ that are not part of any double edge;
    \item exactly $\eta$ out of the $k$ edges in (P4) are incident with an end vertex of a double edge.
\end{enumerate}

\begin{claim}\label{claim:pairs}
    The number of pairs $(H_1,H_2)$ satisfying (P1)--(P5) is
    \begin{align*}
    (e^{-1/2}+O(\xi))& 
    \binom{2i-2t}{\eta}(n-2i)_{\eta}\binom{n-2i-\eta}{2k-2\eta}\frac{(2k-2\eta)!}{2^{k-\eta}(k-\eta)!}\binom{i}{j,\ell,h} 4^j 2^{\ell+h} \\
    &\times\frac{(n-2j-2\ell-2k)!(n-2j-2h-2k)!}{(n/2-j-\ell-k)!(n/2-j-h-k)!2^{n-2j-\ell-h-2k}} ,
    \end{align*}
    where $\xi=d^2/n+(n-2i-2k)^{-1}$.
\end{claim}
Fix a pair $(H_1,H_2)$ satisfying (P1)--(P5).  Then,
\[
\pr(H_1\cup H_2\subseteq \P')=\frac{N(\bfg, X\cup H_1\cup H_2)}{N(\bfg',X)},
\]
where, without loss of generality, with $t=j+\ell+h$,
\[
\bfg=(\underbrace{d-2,\ldots,d-2}_{2j}, \underbrace{d-3,\ldots,d-3}_{2\ell+2h},\underbrace{d-4,\ldots,d-4}_{2i-2t-\eta}, \underbrace{d-3,\ldots,d-3}_{\eta}, \underbrace{d-1,\ldots,d-1}_{2k-\eta},   \underbrace{d-2,\ldots,d-2}_{n-2i-2k+\eta}),
\]
\[
\bfg'=(\underbrace{d-2,\ldots,d-2}_{2i},\underbrace{d,\ldots,d}_{n-2i}).
\]
It is immediate that
\begin{align*}
    m(\bfg)=\frac{n(d-2)}{2}-2i+t+j+k,\quad (\text{which is independent of $\eta$}) 
\end{align*}
and assuming that $i=O(d^2)$ and letting $\alpha=\alpha(k):=2k/n$,
\begin{align*}
    \lambda(\bfg)=\frac{(d-2)(d-3)(1-\alpha)+(d-1)(d-2)\alpha}{2(d-2)+2\alpha}+O(d^3/n),
\end{align*}
and
\begin{align*}
    \mu(\bfg,X\cup H_1\cup H_2)=\frac{(1-\alpha)(d-2)^2+(d-1)^2\alpha/2}{d-2+\alpha}+O(d^3/n).
\end{align*}
Hence,
\begin{align*}
N(\bfg,X\cup H_1\cup H_2) = &\frac{(n(d-2)-4i+2t+2j+2k)!\exp\Big(-\lambda(\bfg)-\lambda(\bfg)^2-\mu(\bfg,X\cup H_1\cup H_2)\Big)}{(n(d-2)/2-2i+t+j+k)!2^{n(d-2)/2-2i+t+j+k}}\\
&\times \frac{1}{(d-2)!^{n-2i-2k+2j}(d-1)!^{2k}(d-4)!^{2i-2t})(d-3)!^{2\ell+2h}} \left(\frac{d-1}{d-3}\right)^{\eta}.
\end{align*}
It follows that $\pr(H_1\cup H_2\subseteq \P')=\rho_2(n,d,\alpha(k))$ where
\begin{align*}
\rho_2(n,d,\alpha(k))    =&\frac{(n(d-2)-4i+2t+2j+2k)!(nd/2-2i)!2^{n-j-t-k}}{(n(d-2)/2-2i+t+j+k)!(nd-4i)!}d^{n-2i} (d-1)^{n-2i-2k}(d-2)^{2i-2j}\\
    &\times (d-3)^{2i-2t}\left(\frac{d-1}{d-3}\right)^{\eta} \exp\left(\frac{(-d^2+5)\alpha^2+(-2d^2+14d-20)\alpha+4(d-2)^2}{4(d-2+\alpha)^2}+O(d^3/n)\right)\\
%    =&\frac{(n(d-2)+2k)!(nd/2-2i)!}{(n(d-2)/2+k)! (nd-4i)!} 2^{n-2i-k} d^{n-2i}(d-1)^{n-2i-2k} \\
 %   &\times \exp\left(\frac{(-d^2+5)\alpha^2+(-2d^2+14d-20)\alpha+4(d-2)^2}{4(d-2+\alpha)^2} +O(d^3/n)\right)\\
    =&  \left(\frac{2n(d-2)-8i+4t+4j+4k}{e}\right)^{n(d-2)/2-2i+t+j+k} \left(\frac{e}{2nd-8i}\right)^{nd/2-2i} \\
    &\times 2^{n-t-j-k} d^{n-2i} (d-1)^{n-2i-2k}(d-2)^{2i-2j}(d-3)^{2i-2t} \left(\frac{d-1}{d-3}\right)^{\eta}\\
    &\times\exp\left(\frac{(-d^2+5)\alpha^2+(-2d^2+14d-20)\alpha+4(d-2)^2}{4(d-2+\alpha)^2} +O(d^3/n)\right)\\
    =&  \left(\frac{2n(d-2)+4k}{e}\right)^{n(d-2)/2+k} (2n(d-2)+4k)^{-2i+t+j} \left(\frac{e}{2nd}\right)^{nd/2-2i}e^{2i} \\
    &\times 2^{n-t-j-k} d^{n-2i} (d-1)^{n-2i-2k}(d-2)^{2i-2j}(d-3)^{2i-2t} \left(\frac{d-1}{d-3}\right)^{\eta}\\
    &\times\exp\left(\frac{(-d^2+5)\alpha^2+(-2d^2+14d-20)\alpha+4(d-2)^2}{4(d-2+\alpha)^2} +O(d^3/n)\right).
\end{align*}

%Therefore,

%\begin{align*}
%\ex Z^2 &=\ex Z+\sum_{k=0}^{n/2-i-1}(1+O(n-2i-2k)^{-1})\frac{(n-2i)!}{2^k k!\sqrt{e\pi(n-2i-2k)/2}} \rho_2(n,d,\alpha(k))\\
%&=\ex Z+\sum_{k=0}^{n/2-i-1}(1+O(n-2i-2k)^{-1}) \sqrt{\frac{2}{e\pi}}(n-2i)!\varphi(k),
%\end{align*}
%where
%\[
%\varphi(k)=\frac{\rho_2(n,d,\alpha(k))}{2^k k!\sqrt{n-2i-2k}}.
%\]
Following exactly the same approach as in~\cite[Lemma 18]{gao2023number}, 
and combining Claim~\ref{claim:pairs},
we found that
\begin{equation}
\ex Y^2=\left(1+O\left(\xi\right)\right) \sqrt{\frac{2\pi}{\bar\delta}} \cdot \Sigma, \label{eq:2nd}
\end{equation}
where $\bar \alpha=1/d$, $\bar k=\lfloor\bar\alpha n/2\rfloor$, and $\bar\delta=\frac{2d}{n}\frac{d(d-2)}{(d-1)^2}$, $\xi=\sqrt{\frac{d}{n}}\log^3 n+\frac{d^4}{n}$,
and
\begin{align*}
\Sigma = & \sum_{j,\ell,h,\eta} e^{-1/2} 
    \binom{2i-2t}{\eta}(n-2i)_{\eta}\binom{n-2i-\eta}{2\bar k-2\eta}\frac{(2\bar k-2\eta)!}{2^{\bar k-\eta}(\bar k-\eta)!}\binom{i}{j,\ell,h} 4^j 2^{\ell+h} \\
    &\times\frac{(n-2j-2\ell-2\bar k)!(n-2j-2h-2\bar k)!}{(n/2-j-\ell-\bar k)!(n/2-j-h-\bar k)!2^{n-2j-\ell-h-2\bar k}}  \rho_2(n,d,\alpha(\bar k))\\
    =& \sum_{j,\ell,h,\eta} e^{-1/2} 
    \binom{2i-2t}{\eta}(n-2i)_{\eta}\binom{n-2i-\eta}{2\bar k-2\eta}\frac{(2\bar k-2\eta)!}{2^{\bar k-\eta}(\bar k-\eta)!}\binom{i}{j,\ell,h} 4^j 2^{\ell+h} \\
    &\times 2 \left(\frac{n-2\bar k}{e}\right)^{n-2\bar k -t-j} e^{-t-j}  \rho_2(n,d,\alpha(\bar k)) (1+O(\xi)).
\end{align*}

First, considering the
summation over $\eta$, we obtain
\begin{align*}
\sum_{\eta} &\binom{2i-2t}{\eta}(n-2i)_{\eta}\binom{n-2i-\eta}{2\bar k-2\eta}\frac{(2\bar k-2\eta)!}{2^{\bar k-\eta}(\bar k-\eta)!}\left(\frac{d-1}{d-3}\right)^{\eta}\\
&= (1+O(\xi)) \frac{(n-2i)_{2\bar k}}{2^{\bar k} \bar k!} \left(1+\frac{\bar \alpha(d-1)}{(1-\bar \alpha)(d-3)}\right)^{2i-2t} \\
&=(1+O(\xi)) \frac{(n-2i)!}{2^{\bar k} \bar k!\sqrt{2\pi(n-2i-2\bar k)}}\left(\frac{e}{n-2i-2\bar k}\right)^{n-2i-2\bar k} \left(\frac{d-2}{d-3}\right)^{2i-2t}\\
&=(1+O(\xi)) \frac{n!}{2^{\bar k } \bar k! n^{2i} \sqrt{2\pi(n-2\bar k)}}\left(\frac{e}{n-2\bar k}\right)^{n-2i-2\bar k} e^{2i} \left(\frac{d-2}{d-3}\right)^{2i-2t}
\end{align*}
Next, by summing over $j,\ell,h$, we obtain that for some function $f(n,d)$ which is independent of $n$,
\begin{align*}
\ex Y^2 = & (1+O(\xi)) f(n,d)  n^{-2i}  \left(\frac{e}{n-2\bar k}\right)^{-2i} e^{2i}  \left(\frac{e}{2nd}\right)^{-2i} e^{2i}  d^{-2i} (d-1)^{-2i} \cdot {\sum}^1,
\end{align*}
where, by noting that $2n(d-2)+4\bar k=nd(d-1)^2/d$, and $n-2\bar k=n(d-1)/d$,
\begin{align*}
    {\sum}^1 = &\sum_{j,\ell,h}  \binom{i}{j,\ell,h} 4^j 2^{\ell+h} (2n(d-1)^2/d)^{-2i+t+j}  2^{-t-j} (d-2)^{2i-2j}(d-3)^{2i-2t}\left(\frac{n(d-1)/d}{e}\right)^{-t-j} e^{-t-j} \\
&\times \left(\frac{d-2}{d-3}\right)^{2i-2t}\\
=&\left(1+g_j+g_{\ell}+g_h\right)^i (2n(d-1)^2/d)^{-2i}(d-2)^{4i},
\end{align*}
where
\begin{align*}
    g_j&=\frac{4 (2n(d-1)^2/d)^2}{(d-2)^2(d-3)^2(n(d-1)/d)^2\cdot 4 ((d-2)/(d-3))^2}=\frac{4(d-1)^2}{(d-2)^4}\\
    g_{\ell}=g_h& =\frac{2(2n(d-1)^2/d)}{(d-3)^2(n(d-1)/d)\cdot 2 ((d-2)/(d-3))^2}=\frac{2(d-1)}{(d-2)^2}.
\end{align*}
Hence, $${\sum}^1=\left(1+\frac{4(d-1)^2}{(d-2)^4}+\frac{4(d-1)}{(d-2)^2}\right)^i \left(\frac{d(d-2)^2}{2n(d-1)^2}\right)^{2i} = \left(\frac{d(d^2-2d+2)}{2n(d-1)^2}\right)^{2i} .$$
Thus,
\begin{align*}
    \ex Y^2=(1+O(\xi)) f(n,d) \left(\frac{d^2-2d+2}{(d-1)^2}\right)^{2i}.
\end{align*}
By~\cite[Theorem 10]{gao2023number} (corrresponding to the case that $i=0$) and~\eqref{eq:FM}, $f(n,d)=(1+1/6d^3+O(d^{-4}+\xi))g(n,d)^2$. 
Hence, $\ex Y^2=(1+1/6d^3+O(d^{-4}+\xi)) (\ex Y)^2$ as desired.

The treatment for $Y_a$ is very similar, and we briefly sketch the argument. Let $\xi$ and $Z_a$ be defined as before.
\begin{align*}
    \ex Y_a = & \ex Z_{a} =\exp(1/4+O(d^4/n))\sqrt{2} \binom{i}{a} 2^{a} (d-1)^{n(d-1)/2-2i+a} d^{-n(d-2)/2} (d-2)^{2i-2a} \\
    =& (1+O(\xi)) g_i(n,d) \binom{i}{a} \left(\frac{2(d-1)}{(d-2)^2}\right)^a.
\end{align*}
where $g_i$ is a function of $n,d,i$ that is independent of $a$.

For the second moment of $Y_a$, consider pairs $(H_1,H_2)$ satisfying (P1)--(P5) as before, but necessarily with $\ell=h=a-j$. Hence,
\begin{align*}
    \ex Y_a^2=(1+O(\xi)) f_i(n,d) \cdot {\sum}^2,
\end{align*}
where
\begin{align*}
{\sum}^2= & \sum_{j}  \binom{i}{j,a-j,a-j} 4^j 2^{2a-2j} (2n(d-1)^2/d)^{2a}  2^{-2a} (d-2)^{-2j}(d-3)^{-2(2a-j)}\left(\frac{n(d-1)/d}{e}\right)^{-2a}\\
&\times e^{-2a}  \left(\frac{d-2}{d-3}\right)^{-2(2a-j)}=\left(\frac{2(d-1)}{(d-2)^2}\right)^{2a} \sum_{j}  \binom{i}{j,a-j,a-j} =\binom{i}{a}^2\left(\frac{2(d-1)}{(d-2)^2}\right)^{2a}.
\end{align*}
By Lemma~\ref{lem:YPairing} (which corresponds to the case that $a=i$), $f_i(n,d)=(1+1/6d^3+O(d^{-4}+\xi))g_i(n,d)^2$. Consequently, $\ex Y_a^2=(1+1/6d^3+O(d^{-4}+\xi))(\ex Y_a)^2$, as desired.
 \qed

The lemma immediately yields the following corollary.

\begin{corollary}
\label{cor2:YPairing} Let $i,j=O(d^2)$, $j\le i$. Let $Y_j$ be the number of perfect matchings $M$ of $P\sim\P^{(i)}(n,d)$ such that $P-M$ has exactly $j$ double edges.
Then, with probability $1-O(\sqrt{\log d}/d^{1.5})$,
%\[
%\ex Y=
%\]    
%and
\[
Y_j=(1+O(\sqrt{\log d}/d^{-1.5}))\ex Y.
\]
\end{corollary}

This yields the following theorem.
\begin{theorem}\label{thm:pairing-couple}
    Let $i=O(d^2)$. There exist a coupling between $P$ and $P'$ where $P$ is distributed as $\P_*(n,d)+\M_1^+$ conditional on having exactly $i$ double edges, and $P'\sim \P^{(i)}(n,d+1)$, and
    $
    \pr(P=P')=1-O(\sqrt{\log d}/d^{1.5}).
    $
\end{theorem}

\proof Construct bipartite grah $B$ on $(A,B)$ where vertices in $A$ are $(P,M)$ where $P\in \P_*(n,d)$, $M$ is a perfect matching on $[n]$ and $P+M$ has exactly $i$ double edges. The vertices in $B$ are pairings in $\P^{(i)}(n,d+1)$. $(P,M)$ and $P'$ are adjacent if $P+M=P'$. It turns out that every vertex in $A$ had degree equal to one, and by Corollary~\ref{cor2:YPairing}, a uniform random $P'\in B$ has degree $(1+O(\sqrt{\log d}/d^{1.5}))\ex Y$ with probability $1-O(\sqrt{\log d}/d^{1.5})$. The existence of the asserted coupling now follows by Corollary~\ref{cor:coupling2}.\qed   

Applying Theorem~\ref{thm:pairing-couple} twice, we obtain the following corollary.
\begin{corollary}
  \label{cor:2matchings}  
    \begin{enumerate}[(i)]
        \item  Let $i=O(d^2)$. There exist a coupling between $P$ and $P'$ where $P$ is distributed as $\P_*(n,d)+\M_2^+$ having exactly $i$ double edges, and $P'\sim \P^{(i)}(n,d+2)$, and
    $
    \pr(P=P')=1-O(\sqrt{\log d}/d^{1.5}).
    $ 
\item For every $i=O(d^2)$, 
$
\pr(\P_*(n,d)+M_1+M_2\ \text{has at most $i$ double edges})$ is at most  $\pr(\P_*(n,d+1)\ \text{has at most $i$ double edges}) +O(\sqrt{\log d}/d^{1.5})$.
    \end{enumerate}
\end{corollary}

\subsection{Proof of Theorem~\ref{thm:P-P2}}
By standard switching argument, the limiting distribution of the number of double edges is $\po((d-1)^2/4)$. Let $i$ and $i'$ denote the number of double edges in $\P_*(n,d+1)$ and $\P_*(n,d+2)$ respectively. Then, there is a coupling such that $i\le i'$ with probability $1-O(\xi)$ where $\xi$. Next, we can couple $\P^{(i)}(n,d+1)$ and $\P^{(i')}(n,d+2)$ so that with probability at least $1-O(\sqrt{\log d}/d^{1.5})$, $\P^{(i)}(n,d+1)\subseteq \P^{(i')}(n,d+2)$ by Theorem~\ref{thm:pairing-couple}. 

By Corollary~\ref{cor:2matchings}, there is a coupling such that with probability at least $1-O(\sqrt{\log d}/d^{1.5})$,  $\P_*(n,d)+\M_2^+=\P_*(n,d+2)$. Combining the two couplings we obtain the required coupling in the theorem.\qed

\subsection{Proof of Theorem~\ref{thm:P-M}} 
Let $\eps>0$ be fixed. By Theorem~\ref{thm:P-P2} there exists $t=t(\eps)>0$ and a coupling such that 
\[
\P_*(n,d)\subseteq \P_*(n,t) + \M_{2(d-t)}^\plus
\]
with probability at least $1-\sum_{j=t}^{d-t}O(\sqrt{t}/t^{1.5}) >1-\eps$, provided that $t(\eps)$ is sufficiently large.

Let $\tau=\tau_n$ be chosen so that $\tau=\omega(1)$ and $\tau=o(d)$. By Theorem~\ref{thm:smalld}(a), there is a coupling such that a.a.s.\
\[
\P_*(n,t)\subseteq \M_{\tau}^\plus.
\]
Combining these two couplings we obtain a new coupling such that with probability at least $1-\eps-o(1)$,
\[
\P_*(n,d)\subseteq \M_{2(d-t)+\tau}^\plus \subseteq \M_{\tau(d)}^\plus
\]
for some $\tau(d)=(2+o(1))d$. \qed

\subsection{Proof of Corollary~\ref{thm:G-M}}

This follows immediately by Theorems~\ref{thm:P-G} and~\ref{thm:P-M}.\qed

\section{Proof of Theorem~\ref{thm:Gnp-M}}

Let $\vO(n,d)$ be the random digraph on vertex set $[n]$ where each vertex $v\in[n]$ picks a random set $N_v$ of $d$ vertices in $[n]\setminus\{v\}$ independently from other choices and adds an arc $vw$ for every $w\in N_v$. Let $\O(n,d)$ be the underlying undirected graph after removing edge orientations in $\vO(n,d)$ and replacing double edges by simple edges. This model is often called the $d$-out random graph in the literature. Let $\G(n,p)$ denote the usual binomial random graph and let $\vG(n,p)$ denote its directed counterpart. That is, $\vG(n,p)$ is a directed graph on vertex set $[n]$ with each possible directed edge $uv$ included with probability $p$ independently from all other choices.

\begin{proposition}\label{prop:out}
There is a continuous and strictly increasing function $f:[1,\infty)\to[0,\infty)$ with $f(1)=0$ and $\lim_{x\to\infty}f(x)=1/2$ such that the following holds for every $\eps>0$.
Let $p=x\log n/n$ with $x=x(n)\ge(1+\eps)$ and $d \le f(x) pn$. Then there is a coupling in which $\O(n,d)\subseteq\G(n,p)$ a.a.s.
\end{proposition}

\proof
The easiest case is when $x=pn/\log n \ge 2+\eps$ for $\eps>0$. Assume $x\ge 2.1$ for simplicity. Let $\vec G$ be distributed as $\vG(n,\vec p)$ with $\vec p = 1 - (1-p)^{1/2} \sim p/2$. Then, the undirected graph $G$ obtained from $\vec G$ by ignoring the orientation of the edges (and replacing double edges by simple edges) is distributed as $\G(n,p)$.
%We want an a.a.s.\ lower bound on the minimum out-degree of $\vec G$.
Since $\vec p \sim (x/2)\log n/n$ and $x/2\ge 1.05$, a standard first moment calculation (see~e.g.~Lemma~3.2 in~\cite{PPSS15}) shows that the minimum out-degree $\delta^+=\delta^+(\vec G)$ of $\vec G$ is a.a.s.\ at least $(1+o(1))f_1(x/2)(x/2)\log n$, where
\[
f_1(x) = - \frac{(x-1)/x}{W(-(x-1)/(xe))}
\]
and $W$ is the lower branch of the Lambert $W$ function.
Let $\mathcal F$ be such event.
Note that $f_1(x)$ is strictly increasing in $[1,\infty)$, $f_1(1)=0$ and $\lim_{x\to\infty}f(x)=1$.

Conditional on $\mathcal F$, by picking $d \le (1-\eps)f_1(x/2)(x/2)\log n \le \delta^+$ random out-neighbours of each vertex $v$ in $\vec G$ independently, we obtain a copy of $\vO(n,d)$, so $\vO(n,d) \subseteq \vG(n,\vec p)$ and, ignoring orientations, $\O(n,d) \subseteq \G(n,p)$.

Now, we consider the case $p=(1+\eps)\log n/n$, where $0<\eps\le\eps_0$ and $\eps_0$ is assumed to be sufficiently small for some inequalities to hold. Let $d\le \eps^2\log n$. We will and find a coupling in which a.a.s.\ $\O(n,d) \subseteq \G(n,p)$.
Let
\[
f_2(x) = \begin{cases}
\frac{\eps_0^2}{1.1^2}\frac{(x-1)^2}{x} & 1\le x\le 2.1
\\
f_1(x)/2 & x> 2.1.
\end{cases}
\]
Note $f_2(x)pn \le (x-1)^2\log n = \eps^2\log n$ for $x=1+\eps\le 1+\eps_0$, and $f_2(x)pn \le \eps_0^2\log n$ for $1+\eps_0 \le x\le 2.1$. Let $f:[1,\infty)\to[0,\infty)$ be any function with $f(1)=0$ and $f(x)<f_2(x)$ in $(1,\infty)$.
Then, by monotonicity and by the easy case above, the coupling holds for all $d\le f(x)pn$ and $p\ge (1+\eps)\log n/n$.
Since $\lim_{x\to\infty}f_2(x)=1/2$ and $f_2$ is continuous and strictly increasing in each of the two pieces of the domain, one can trivially build function $f$ to be continuous and strictly increasing on $[1,\infty)$ and satisfy $\lim_{x\to\infty}f(x)=1/2$. This function satisfies the conditions of the statement.

It only remains to find the promised coupling for $p=(1+\eps)\log n/n$ and $d\le \eps^2\log n$, where $0<\eps\le \eps_0$.
We build three random graphs independently: $\vE_1$,  $E_2$ and  $\vE_3$. Note that $\vE_1$ and $\vE_3$ are directed, while $E_2$ is undirected. $\vE_1$ is distributed as $\vG(n,\vec p_1)$ with $\vec p_1=1-(1-p_1)^{1/2}\sim p_1/2$ and $p_1=(\eps/2)\log n/n$.
$E_2$ is distributed as $\G(n,p_2)$ with $p_2=(1+\eps/2)\log n/n$.
$\vE_3$ is distributed as $\vG(n,\vec p_3)$ with $\vec p_3=1-(1-p_2)^{1/2}\sim p_2/2$.
Note that if $E_i$ ($i\in\{1,3\}$) is obtained from $\vE_i$ by forgetting the orientation of the edges, then it is distributed as $\G(n,p_i)$ with $p_1$ defined above and $p_3=p_2$. %Therefore, there is no hope of embedding $E_1\cup E_2\cup E_3$ inside of $\G(n,p)$ since the union contains around $p_1+p_2+p_3=(2\alpha(c)+2c)\log n$ which is larger than $p$. So our approach will be to consider all edges of type 1 together with some of type 2 and some of type 3 in a way that we can build the desired coupling.

Given a vertex $v$ in an oriented hypergraph, its out-edges are the edges that have $v$ as a tail. Out-degrees, out-neighbourhoods and other related notions are defined in
the obvious way. We say $v\in[n]$ is good if its out-degree is at least $d$ in $\vE_1$.
Otherwise we call vertex $v$ bad.

\begin{lemma}\label{lem:lowdegree}
Let $d\sim\beta\log n$ and $p\sim\alpha\log n/n$.
The number of vertices in $\G(n,p)$ of degree at most $d$ is a.a.s.\ at most $n^{\theta+o(1)}$, where $\theta = 1 - \alpha + \beta\log(e\alpha/\beta)$. The same holds true for the number of vertices in $\vG(n,p)$ of out-degree at most $d$.
\end{lemma}
\proof
In either case, the number of such vertices is $\text{Bin}(n-1,p)$, so its expectation is at most
\[
n \sum_{k=0}^{\lfloor d\rfloor} \binom{n-1}{k}p^k(1-p)^{n-1-k} \le
(1+o(1)) n \left(\frac{enp}{d}\right)^de^{-pn} \sim
n^{\theta+o(1)}.
\]
The result follows immediately from Markov's inequality.
\qed

Note that with $\alpha=\eps/4$ and $\beta=\eps^2$ in Lemma~\ref{lem:lowdegree}, we get that $\theta = 1 - \eps/4 + \eps^2\log(e/(4\eps)) < 1-\eps/8$ (for sufficiently small $\eps_0$), so a.a.s\ there are at most $n^{1-\eps/8}$ bad vertices. We call this event $\mathcal F_1$.

Let $V_G$ be the set of good vertices and $V_B=[n]\setminus V_G$ the set of bad ones. Let $E_2'=\{e\in E_2 : |e\cap V_B|\le 1\}$, that is, the set of edges in $E_2$ that are incident with at most one bad vertex. We can orient each edge in $E_2'$ with exactly one bad vertex by making that bad vertex the tail, and edges with no bad vertices are oriented following any arbitrary deterministic rule. The resulting oriented hypergraph is denoted $\vec E_2'$.
Let $\mathcal F_2$ be the event that every bad vertex has degree at least $d+1$ in $E_2'$ (or equivalently out-degree at least $d+1$ in $\vec E_2'$). The reason why we ask degree $d+1$ instead of $d$ is technical and will become apparent later. Again by Lemma~\ref{lem:lowdegree} with $\alpha=1+\eps/2$ and $\beta=\eps^2$, we have that $\theta=-\eps/2+\eps^2\log(e(1+\eps/2))/\eps^2) < -\eps/4 <0$ (for sufficiently small $\eps_0$), so $\mathcal F_2$ holds a.a.s.

Then define $\vec E_3' = \{e\in \vec E_3 : |e\cap V_B| > 1\}$ and let $E_3'$ be the undirected version of $\vec E_3'$ resulting from ignoring edge orientations. Edges in $E_2'$ appear with probability $p_2$ among those potential edges containing at most one bad vertex. Edges in $E_3'$ appear with probability $1 - (1 - p_3)^2 = p_2$ among the other potential edges that are not elegible for $E_2'$. In other words, even though each one of $E_2'$ and $E_3'$ depend on the set of good vertices (and thus on $E_1$), $E_2'\cup E_3'$ is distributed as $\G(n,p_2)$ and is independent of $E_1$. Then $E_1 \cup E_2' \cup E_3'$ is distributed as $\G(n,p_1+p_2-p_1p_2)$ which can be trivially included inside of $\G(n,p)$ since $p_1 + p_2 - p_1p_2 \le p_1 + p_2 = p$.
Hence, all it remains to show is that in the event that $\mathcal F_1\cap \mathcal F_2$ holds, we can find a copy of $\vO(n,d)$ contained in $\vec E_1\cup\vec E_2'\cup\vec E_3'$, which after forgetting orientations we just showed is contained in $\G(n,p)$. We denote such copy of $\vO(n,d)$ by $\vO$. If $\mathcal F_1\cap \mathcal F_2$ fails, just pick $\vO$ independently from everything else so the coupling fails, but this occurs with probability $o(1)$.

Expose the out-degrees of $\vec E_1$ and the degrees of $E_2'$ (but nothing else), and suppose that $\mathcal F_1\cap \mathcal F_2$ holds. In particular, we know which vertices are good or bad. For every vertex $v\in[n]$, we want to pick a set of $d$ out-edges of $v$ uniformly at random and independently of all other choices. These will be the out-edges of $v$ in $\vO$. If $v$ is good, this is straightforward, since $v$ has out-degree $d_v\ge d$ in $E_1$ and all sets of $d_v$ out-neighbours are equally likely and independent of those of other vertices, so we can just pick a random subset of size $d$ among those.

If $v$ is bad, picking its $d$ out-edges in $\vO$ is trickier. We will be using edges in $\vec E_2'\cup\vec E_3'$ with tail $v$. In particular, we will only use edges in $\vec E_2'$ with exactly one bad vertex, and ignore those with no bad vertices. Recall that edges of type 2 were originally unoriented, but we oriented them in $\vec E_2'$ so that their only bad vertex (if they have one) is the tail. We will also use edges in $\vec E_3'$ that have tail $v$. Most bad vertices will not pick vertices from $\vec E_3'$, but some will. Note that the edge densities of $\vec E_2'$ and $\vec E_3'$ are different since each out-edge of $v$ with only one bad vertex appears in $\vec E_2'$ with probability $p_2$ (before conditioning on degrees of $E_2$) while each out-edge of $v$ with more bad vertices appears in $\vec E_3'$ with probability $p_3\sim p_2/2$. This is a delicate issue that complicates things and is possibly the crux of the argument. We cannot make the density of $\vec E_2'$ smaller a priori since we want to make sure that bad vertices have enough out-edges and we cannot make the density of $\vec E_3'$ larger since otherwise the coupling with $\G(n,p)$ would not work.

To make the argument slightly simpler, instead of having each bad vertex $v$ pick a set of $d$ distinct out-edges, we will have it sample $d+1$ out-edges with replacement from all possible edges from $v$ to $[n]\setminus \{v\}$. Of course, the same out-edge could be picked multiple times but with probability $1-O(\log^4/n)$ every bad edge picks at most one repeated pair of out-edges. Call this event $\mathcal F_3$. In other words, if $\mathcal F_3$ holds, then every bad edge will pick a set of either $d$ or $d+1$ different out-edges. 

Let us proceed to do this carefully. Let $v$ be a bad vertex. Let $X = X_v$ be the set of all possible out-edges of $v$ in $\vec E_2'$, i.e. edges with tail at $v$ and a good vertex as head. Note $x = |X| = |V_G| \sim n$. Let $Y = Y_v$ be the set of all possible out-edges of $v$ in $\vec E_3'$, i.e.~edges with tail at $v$ and head at another bad vertex. Clearly, $y = |Y| = n-1-x \le n^{-\eps/8}x$ by $\mathcal F_1$. Since we are also assuming that $\mathcal F_2$ holds, the out-degree of $v$ in $\vec E_2'$ is $s_v \ge d+1$, so the set $S = S_v$ of out-edges of $v$ in $\vec E_2'$ is a random subset of $X$ of size $s_v \ge d+1$. Let $S'$ be a random subset of $S$ of size exactly $d+1$ (which is still uniformly chosen among all subsets of $X$ of that size). Also, let $T = T_v$ be the set of out-edges of $v$ in $\vec E_3'$. Now $t = |T|$ is distributed as $\text{Bin}(y, \vec p_3)$. Our goal is to find a coupling $U\subseteq S' \cup T$ where $U = \{u_1,\ldots, u_{d+1}\}$ is obtained by uniformly and independently sampling $d+1$ random elements of $X\cup Y$ (where recall $X\cup Y$ is the set of all possible out-edges of $v$) with replacement (i.e.~$|U|\le d+1$ and could be strictly smaller if the same edge is sampled twice).

To do this, for each $i = 1,\ldots,d+1$, let $a_i$ be an integer in $[x+y] = \left[n-1\right]$ chosen uniformly at random (and independently of other choices). We can think of the edges in $X\cup Y$ as having distinct labels in $[x + y]$ so we pick $u_i$ by first picking its label $a_i$, and later revealing what edge corresponds to that label. However, all we know so far is that edges in $Y$ have labels in $[y]$ and edges in $X$ have labels in $[x + y] \setminus [y]$. The label of each individual edge is still undecided. The reason to do this is to see: 1) how many different edges are in $U$ after possible edge repetitions, and 2) how many of these fall in $X$ or $Y$ . If $a_i \le y$, that means that $u_i$ will land in $Y$ (which has probability $y/(x+y) \le n^{-\eps/8}$). Otherwise, it lands in $X$. Let $A = \{a_1,\ldots,a_{d+1}\}$, $A_X = A\setminus[y]$ and $A_Y = A\cap[y]$. Then $u = |U| = |A|$, $u_X = |U \cap X| = |A_X|$ and $u_Y = |U \cap Y| = |A_Y|$. Let $\ell$ be the number of indices $i = 1,\ldots,d+1$ such that $a_i\in[y]$ (or equivalently $u_i\in Y$). Clearly, $\ell$ has distribution $\text{Bin}(d+1, y/(x + y))$ and moreover $u_Y \le \ell$ by construction. If $u_Y \le t$, then we can assign the $u_Y$ labels in $A_Y$ to a random set of edges in $T$. Similarly, we can always assign the $u_X \le d+1 = |S'|$ labels in $A_X$ to a random set of edges in $S'$. Doing so, the $u_1,\ldots,u_{d+1}$ are i.i.d.\ uniformly chosen from $X\cup Y$ as desired. For all of this to work, it is enough to find a coupling in which $\ell\le t$, or in other words $\text{Bin}(d+1, y/(x + y)) \le \text{Bin}(y, \vec p_3)$. Ignoring $(1 + o(1))$ factors in the parameters, all we need is a coupling
\begin{equation}\label{eq:coupling}
\text{Bin}( \eps^2\log n, 1/n^{\eps/8} ) \le \text{Bin} \left( n^{1-\eps/8}, \frac{1 + \eps/2}{2} \frac{\log n}{n} \right). 
\end{equation}
The coupling in~\eqref{eq:coupling} exists (see~e.g.~Lemma~6 in~\cite{FP25} with $N = d+1 \sim \eps^2\log n$, $P = 1/n^{-\eps/8}$, $K = \frac{2}{1+\eps/2}\frac{n}{n^{\eps/8}\log n}$ and $L = \frac{1+\eps/2}{2\eps^2}$).

Summarizing, for each bad vertex $v$ we were able to sample $u_1,\ldots,u_{d+1}$ uniformly at random from the set of all $n-1$ out-edges from $v$ and indepdently from all other vertices in a way that $u_1,\ldots,u_{d+1}$ are also in $\vec E_2'\cup \vec E_3'$ (assuming $\mathcal F_1$ and $\mathcal F_2$). Also, conditional on $\mathcal F_3$, we have $|\{u_1,\ldots,u_{d+1}\}|\in\{d,d+1\}$ and thus we can sample a subset of $d$ out-edges from $v$ uniformly at random with the same properties. All these edges (after ignoring orientations) are in $E_2'\cup E_3'$ and thus in $\G(n,p)$.
This gives the desired coupling and completes the proof of the proposition.\qed

We now proceed to the proof of Theorem~\ref{thm:Gnp-M}. Let $p=x\log n/n$ and $d\le f(x)\log n$ for $x=x(n)\ge 1+\eps$ and function $f$ as in the statement of Proposition~\ref{prop:out}. Then there is a coupling in which $\O(n,d) \subseteq \G(n,p)$ a.a.s.\ In turn, we can always couple $\lfloor d/2\rfloor$ independent copies $\O_1,\ldots,\O_k$ of $\O(n,2)$ so that $\O_1\cup\cdots\cup\O_k \subseteq \O(n,d)$. It was shown by Frieze~\cite{Frieze86} that a.a.s.\ $\O(n,2)$ contains a perfect matching (for even $n$). As a result, it follows from Markov's inequality that a.a.s.\ all but a $o(1)$ fraction of the $\O_i$ contain a perfect matching. Pick one random perfect matching $M_i$ from each $\O_i$ that contains some. These perfect matchings are uniform and i.i.d.\ by construction. Hence, a.a.s.\ $\G(n,p)$ contains the union of $(1+o(1))d/2$ uniform independent perfect matchings. This finishes the proof of the theorem with $g(x)=f(x)/2$.\qed

\section{Proof of Theorem~\ref{thm:KM}}

By taking the graph complement, part (a) follows immediately from parts (b,c,d).
To prove parts (b,c,d) we consider two cases in terms of the pairity of $n$.

\subsection{Even $n$}
Suppose that $n$ is even. If $d=O(1)$, then let $\omega(1)=\tau=o(\log n)$ and  by Theorem~\ref{thm:smalld}(b), there is a coupling such that a.a.s.\ $\G(n,d)\subseteq \M_{\tau}^+$, which implies that $\G(n,d)\subseteq \M_{\tau}^\cup$. By Theorem~\ref{thm:Gnp-M}, there is a coupling such that a.a.s.\ $\M_{\tau}^\cup \subseteq \G(n,(1+\eps)\log n/n)$.

Suppose $\omega(1)=d=o(\log n)$. By Corollary~\ref{thm:G-M}, there exists a coupling such that a.a.s.\ $\G(n,d)\subseteq \M_{\tau(d)}^\cup$ for some $\tau(d)=(2+o(1))d$. By Theorem~\ref{thm:Gnp-M}, there is a coupling such that a.a.s.\ $\M_{\tau(d)}^\cup \subseteq \G(n,(1+\eps)\log n/n)$. Now part (c) follows by combining the above two cases.

For parts (b), the embedding of $\G(n,p_*)$ into $\G(n,d)$ was already shown in previous works, e.g.~\cite{kim2004sandwiching}. To embedd $\G(n,d)$ into $\G(n,p^*)$, we observe that there is a coupling such that $\G(n,d)\subseteq \M_{\tau(d)}^\cup$ for some $\tau(d)=(2+o(1))d$. Then, there is a coupling such that $\M_{\tau(d)}^\cup \subseteq \G(n,p^*)$ by Theorem~\ref{thm:Gnp-M}. 

Finally for part (d), the embedding of $\G(n,d)$ into $\G(n,p^*)$ is exactly the same as in part (b).

\subsection{Odd $n$}
Suppose that $n$ is odd. In this case, $d$ is necessarily even. Couple $\G(n,d)$ and $\G(n,p)$ as follows.

\begin{enumerate}[(i)]
    \item Couple $(G_1, G_2)$ where $G_1 \in \G(n-1, d-1)$, and $G_2 \in \G(n-1,p')$ where $p'=p^*$ in Theorem~\ref{thm:Gnp-M} for even $n-1$. In particular, in case (c), $p^*$ was chosen to satisfy the assertion for $\eps/3$.
    \item generate $\hat G$ by uniformly choosing $d$ vertices $x_1,… x_d$ in $G_1$, join them to the new vertex $n$, and then adding an independent uniform random perfect matching $M$ on the remaining $n-1-d$ vertices disjoint from $G_1$.
    \item Let $p=1-(1-p')(1-(1-\eps/3)\log n/n)$, and $D\sim \Bin(n-1,p)$. Let $z=\max\{0,D-d\}$.
    Construct $G^*$ by joining vertex $n$ with $x_1,…, x_d$, and $z$ other vertices uniformly chosen from $[n-1]\setminus \{x_1,\ldots,x_d\}$, and then taking the union with $G'$ where $G'\sim\G(n-1-d,(1+\eps/3)\log n/n)$ is independent with $D$ and the choices for the neighbours of vertex $n$.
    
\end{enumerate}

Now the existence of the coupling such that a.a.s.\ $\G(n,d)\subseteq \G(n,p)$ is implied by the following lemma, and the observation that $p\sim p'+(1+\eps/3)\log n/n$.
\begin{lemma}
\begin{enumerate}[(a)]
    \item There is a coupling $(\hat G, G_d)$ such that $G_d\sim \G(n,d)$ and $\pr(\hat G=G_d)=1-o(1)$.
    \item $G^* \sim \G(n,p')$.
    \item $(M,G’)$ can be coupled so that a.a.s. $M\subseteq G’$.
\end{enumerate}    
\end{lemma}

\proof Part (b) is obvious.
Part (a) follows by a similar proof as Theorem~\ref{thm:P-G}, which we briefly sketch as follows. Note that for every realisation $G_1$ and $x_1,\ldots,x_d$, the number of choices of $M$ is highly concentrated. On the other hand, for every $d$-regular $G$ on $[n]$, the number of ways to specify $G_1$ and $x_1,\ldots,x_d$ is exactly the number of perfect matchings on the subgraph of $G$ inducced by $[n]\setminus N_G[\{n\})]$, which is concentrated following the similar proof as Lemma~\ref{lem:YPairing}. 

For part (c) it suffices to show that a.a.s.\ $G'\setminus G_1$ has a perfect matching. We already know that there is a coupling such that $G_1\subseteq \G(n-1,p^*)$. Hence, there is a coupling such that $\G(n-1-d, \hat p)\subseteq G'\setminus G_1$ a.a.s.\ where $\hat p\ge (1+\eps/4)\log n/n$, and thus a.a.s.\ $G'\setminus G_1$ has a perfect matching. \qed

%(a). Show that $\hat G$ is close to $\G(n,d)$
%(b). Generate $\hat G^*$ by joining vertex $n$ with $x_1,…, x_d$, and a few more; then take an independent union with $G’~\G(n-1,2\log n/n)$.
%(c). Show that $(M,G’)$ can be coupled so that a.a.s. $M\subseteq G’$. This follows from the fact that a.a.s.\ $G’$ has a perfect matching, and all matchings of the same order are isomorphic by symmetry.

\end{document}